\documentclass[a4paper,10pt]{article}
\usepackage[utf8]{inputenc}
\usepackage{amsmath,amsfonts,amssymb,amscd,amsthm,xspace, graphicx}
\usepackage[normalem]{ulem}
\usepackage[mathscr]{eucal}
\usepackage{color}
\usepackage{ulem}
\usepackage[table]{xcolor}

\usepackage[pagebackref,colorlinks,linkcolor={blue},citecolor={red}]{hyperref}
\newtheoremstyle{plain}{\topsep}{\topsep}{\slshape}{}{\bfseries}{.}{.5em}{}
\newtheorem{theorem}{Theorem}
\newtheorem{lemma}[theorem]{Lemma}
\newtheorem{corollary}[theorem]{Corollary}

\theoremstyle{remark}
\newtheorem{remark}[theorem]{Remark}
\numberwithin{equation}{section}

\def\div{\hbox{\rm div}\,}
\newcommand{\meas}{\mathop{\mathrm{meas}}}
\def\loc{{\mathrm loc}}
\def\R{{\mathbb R}}

\def\ve{{\mathbf v}}
\def\me{{m}}
\def\aa{{\mathbf a}}
\def\bb{{\mathbf b}}
\def\ee{{\mathbf e}}
\def\we{{\mathbf w}}
\def\e{{\varepsilon}}
\def\eee{{\tilde\varepsilon}}
\def\fe{{\mathbf f}}

\def\Xint#1{\mathchoice
{\XXint\displaystyle\textstyle{#1}}%
{\XXint\textstyle\scriptstyle{#1}}%
{\XXint\scriptstyle\scriptscriptstyle{#1}}%
{\XXint\scriptscriptstyle\scriptscriptstyle{#1}}%
\!\int}
\def\XXint#1#2#3{{\setbox0=\hbox{$#1{#2#3}{\int}$ }
\vcenter{\hbox{$#2#3$ }}\kern-.6\wd0}}

\def\dashint{\Xint-}

\begin{document}

\title{Existence and uniqueness for plane stationary Navier{--}Stokes flows with compactly supported force}
\author{Julien Guillod\footnotemark[1] \and Mikhail Korobkov\footnotemark[2]  \and Xiao Ren\footnotemark[3]}
\renewcommand{\thefootnote}{\fnsymbol{footnote}}
\footnotetext[1]{Sorbonne Université, CNRS, Université de Paris, Inria, Laboratoire Jacques-Louis Lions (LJLL), 75005 Paris, France. email: julien.guillod@sorbonne-universite.fr}
\footnotetext[2]{
School of Mathematical Sciences, Fudan University, Shanghai 200433, P.R.China; and Sobolev Institute of Mathematics, pr-t Ac. Koptyug, 4, Novosibirsk, 630090, Russia.  email: korob@math.nsc.ru}
\footnotetext[3]{School of Mathematical Sciences,
Fudan University, Shanghai 200433, P. R.China. email: xiaoren18@fudan.edu.cn}

\maketitle

\begin{abstract}We study the stationary Navier--Stokes equations in the whole plane with a compactly supported force term and with a prescribed constant spatial limit.  Prior to this work, existence of solutions to this problem was only known under special symmetry and smallness assumptions.

In the paper we solve the key difficulties in applying  Leray's {\it invading domains method} and, as a consequence, prove the existence of $D$-solutions in the whole plane for arbitrary compactly supported force. The boundary condition at infinity are verified in two different scenarios: (I) the~limiting velocity is sufficient large with respect to the external force, (II) both the total integral of force and the~limiting velocity vanish. Hence, our method produces large class of new solutions with prescribed spatial limits. Moreover, we show the uniqueness of $D$-solutions to this problem in a perturbative regime.

The main tools here are two new estimates for general Navier-Stokes solutions, which have rather simple forms. They control the difference between mean values of the velocity over two concentric circles in terms of the Dirichlet integral in the annulus between them. 
\end{abstract}

\renewcommand{\thefootnote}{\arabic{footnote}}

\section{Introduction}

\subsection{Motivation and main results}
We study the stationary Navier--Stokes equations in the whole plane $\mathbb{R}^2$ driven by a force term:
\begin{equation}  \label{NSE} 
\left\{
\begin{aligned}
 &- \Delta \mathbf{w} + (\mathbf{w} \cdot \nabla) \mathbf{w} + \nabla p = \mathbf{f}, \\
 & \nabla \cdot \mathbf{w} = 0, \\
 & \mathbf{w}(z) \to\mathbf{w}_\infty=\lambda \mathbf{e}_1 \ \ \text{as}\ \  |z| \to \infty.
\end{aligned}
\right.
\end{equation}
Here, $\lambda\in\mathbb{R}_{\ge 0}$ is a physical parameter specifying the ``boundary" condition at spatial infinity. One major mathematical challenge in the study of the stationary Navier-Stokes equations is  the existence of solutions in two-dimensional unbounded domains along with the characterization of their asymptotic behaviours. In the case when the domain $\Omega$ is exterior, \emph{i.e.}, $\Omega = \mathbb{R}^2 \setminus \overline{U}$ with $U$ bounded, one often takes $\mathbf{f} = 0$ and homogeneous Dirichlet boundary condition on $\partial \Omega$,  so that the system describes stationary flows past rigid obstacles. For comparison,  we also present such an exterior domain problem here,
\begin{equation}  \label{NSE-ext}
\left\{
\begin{aligned}
 & -\Delta \mathbf{w} + (\mathbf{w} \cdot \nabla) \mathbf{w} + \nabla p = 0 \ \ \mathrm{in} \ \Omega,\\
 & \nabla \cdot \mathbf{w} = 0 \ \ \mathrm{in} \ \Omega,\\
 & \mathbf{w}|_{\partial \Omega} = 0, \\
 & \mathbf{w}(z) \to\mathbf{w}_\infty=\lambda \mathbf{e}_1 \ \ \text{as}\ \  |z| \to \infty.
\end{aligned}
\right.
\end{equation} 
In this paper, we focus on another physically important case where the unbounded domain is simply the whole plane $\mathbb{R}^2$. The aim of this paper is to tackle some key difficulties in the $\mathbb{R}^2$ case and establish new existence and uniqueness results. 

Our starting point is to apply Leray's invading domains method  proposed in 1933 \cite{Leray}. Leray's original idea was intended for the exterior domain problem \eqref{NSE-ext} --- one first solves the system on large bounded domains $\Omega_k := \Omega \cap B_{R_k} (k=1,2,3,\cdots)$ with an increasing sequence of radii $R_k \to \infty$, and then takes the limit $k \to \infty$. The boundary condition \eqref{NSE-ext}$_4$ at infinity will now be imposed on the outer boundary.

\begin{equation}
\label{NSE-k-ext}
\left\{\begin{array}{r@{}l}
- \Delta{{\mathbf w}_k}+({\mathbf w}_k\cdot\nabla){\mathbf w}_k+ \nabla p_k  & {} ={\bf 0}\qquad \hbox{\rm in } \Omega_k, \\[2pt]
\nabla \cdot {{\mathbf w}_k} & {} =0\,\qquad \hbox{\rm in } \Omega_k,  \\[2pt]
{{\mathbf w}_k} & {} = \mathbf0\,\qquad \hbox{\rm on } \partial\Omega,  \\[2pt]
{\mathbf w}_k & {}=\we_\infty\;\quad\mbox{for \ }|z|=R_k.
 \end{array}\right.
\end{equation}
In \cite{Leray}, Leray showed that the above invading domain problems produce a sequence of solutions with uniformly bounded Dirichlet integral, \emph{i.e.}, 
$$\int_\Omega |\nabla \mathbf{w}_k|^2 \le C,$$ 
with $C$ independent of $k$. Since $\mathbf{w}_k$ vanishes on $\partial \Omega$, by Sobolev embedding the Dirichlet integral controls local $L^p$ norms of $\mathbf{w}$ for any $p<\infty$. Hence, weak limits of $\mathbf{w}_k$ are well-defined, and one obtains solutions in $\Omega$ with bounded total Dirichlet integral\footnote{Solutions of bounded Dirichlet integral are often called $D$-solutions.}. These solutions are now referred to as Leray's solutions. Whether Leray's solutions achieve the desired spatial limit $\mathbf{w}_\infty$ has been open for nearly 90 years\footnote{This major open problem is only for two dimensions. The main difficulty in 2d is that the Dirichlet energy alone is not sufficient to control the behaviour of functions at infinity. In three dimensions, Leray's solutions do achieve the correct limiting values.}. Recently, it was settled in the small Reynolds number case \cite{KR21new}. The classical papers of Gilbarg--Weinberger~\cite{GW1, GW}, Amick\cite{Amick}  and many recent works, \emph{e.g.}, \cite{GaldiS, S, KPR, KPR2} are devoted to the asymptotic properties of Leray's solutions, or more generally, of $D$-solutions in exterior domains. In particular, the second author with Pileckas and Russo showed that $D$-solutions  in exterior domains always have uniform spatial limits \cite{KPR}. For more backgrounds and details on the exterior domain problem, we refer the readers to the papers \cite{KPR20, KR21} and the book of Galdi \cite{G}. 

However, Leray's original arguments left out the whole plane case. Although it is easy to formulate the invading domain problems here which appear similar to \eqref{NSE-k-ext},
\begin{equation}
\label{NSE-k}
\left\{\begin{array}{r@{}l}
- \Delta{{\mathbf w}_k}+({\mathbf w}_k\cdot\nabla){\mathbf w}_k+ \nabla p_k  & {} ={\bf f}\qquad \hbox{\rm in }  B_{R_k}, \\[2pt]
\nabla \cdot{{\mathbf w}_k} & {} =0\,\qquad \hbox{\rm in } B_{R_k},  \\[2pt]
{\mathbf w}_k & {}=\we_\infty\;\quad\mbox{for \ }|z|=R_k.
 \end{array}\right.
\end{equation}
there are new essential obstructions arising in the analysis, as already pointed out in \cite{GuilW2}. Below, we summarize all the key obstructions that must be tackled. Note that (a) and (b) are inherent to the whole plane problem \eqref{NSE}, while (c) are shared for both \eqref{NSE} and \eqref{NSE-ext}. 
\begin{itemize}
\item[(a)] It is difficult to prove that the solutions $\mathbf{w}_k$ to \eqref{NSE-k} have uniformly bounded Dirichlet integrals.  It should be mentioned that, when the total force $\mathcal{F} = \int_{\mathbb{R}^2} \mathbf{f}$ vanishes, such a difficulty is absent by the following simple argument. We can write $\mathbf{f} = \nabla \cdot \mathbb{F}$ for some tensor $\mathbb{F} \in L^2(\mathbb{R}^2)$ (see \cite[Lemma 3.6]{GuilW2} or Lemma \ref{lem-tensorF} below). Then, we test \eqref{NSE-k}$_{1}$ with $\mathbf{w}_k - \lambda \mathbf{e}_1$ and integrate by parts to get the energy equality
\begin{equation}
\int_{B_{R_k}} |\nabla \mathbf{w}_k|^2 = \int_{B_{R_k}} \mathbb{F} : \nabla \mathbf{w}_k  
\end{equation}
where $\mathbb{F} : \nabla \mathbf{w}_k$ stands for $\sum_{i,j} \mathbb{F}_{ij} \partial_j w_{k,i}$. Using H\"older's inequality we get a nice bound on the Dirichlet integral,
\begin{equation*}
\int_{B_{R_k}} |\nabla \mathbf{w}_k|^2\le \int_{\mathbb{R}^2} |\mathbb{F}|^2.
\end{equation*}
However, for a general $\mathbf{f}$, such a direct estimate is unavailable.

\item[(b)]
 Moreover, even if the uniform boundedness in (1) is assumed, it is not clear whether the local $L^p$ norms of $\mathbf{w}_k$ are uniformly bounded since the Dirichlet integral controls only the derivative of $\mathbf{w}_k$ and the condition \eqref{NSE-k}$_3$ is imposed on distant outer boundaries. 

\item[(c)]
If both (a) and (b) are solved, one may define $\mathbf{w}_L$ as the weak limit of some subsequence of $\mathbf{w}_k$. Then there is still one more difficult question, that is, 
\begin{equation}\label{des-eq}
\mathrm{do \ we \ have} \ \ \lim_{|z| \to \infty} \mathbf{w}_L(z) = \mathbf{w}_\infty ?
\end{equation}
\end{itemize}

Now, we are ready to introduce the main results of the paper for the systems \eqref{NSE} and \eqref{NSE-k}. 
 We shall assume that the force  $\mathbf{f}$ is compactly supported, \emph{i.e.}, $\mathrm{supp}\, \mathbf{f} \subseteq B_R$ for some $R > 0$. Hence, the theories for $D$-solutions in exterior domains without force can be applied in our situation. We also assume a minimal regularity of $\mathbf{f}$, that is, $\mathbf{f} \in H^{-1}(B_{2R})$, see Section \ref{sec-not}. 

First of all, fortunately we were able to completely solve (a) and (b), that is, we proved  the uniform estimates $\|\we_k\|_{L^q(B_1)}+\|\nabla\we_k\|_{L^2(B_{R_k})}\le C$ (see Theorem~\ref{th-uniform-bound}). As an immediate corollary, the Leray solutions~$\we_L$ are well-defined as weak limits of $\we_k$. Hence we have constructed solutions to \eqref{NSE}$_{1,2}$ for arbitrary force $\mathbf{f}$ in the whole plane. Second, we study (c) and justify the identity~(\ref{des-eq}) for each of the following two scenarios.
\begin{enumerate}
\item[(I)] the limiting speed $\lambda$ is large enough with respect to the norm of force $\|\fe\|_{H^{-1}(B_{2R})}$ (see Theorem~\ref{th-s1});

\item[(II)] the total integral of force and the limiting speed $\lambda$ are both zero (see Theorem~\ref{th-s2}).
\end{enumerate}

Third, we prove the uniqueness of $D$-solutions to~(\ref{NSE})  in a perturbative regime (see 
Theorem~\ref{th-uniq}). We emphasize that the conditions in Theorem~\ref{th-s1} and Theorem~\ref{th-uniq} are very different in nature. In particular, for fixed $R$, the condition \eqref{eq-uni-1} implies that $\|\mathbf{f}\|_{H^{-1}(B_{2R})}$ is smaller than an absolute constant independent of $\lambda$, while scenario (I) allows arbitrarily large external force as long as $\lambda$ is chosen even larger.

With the help of Steps 4-5 in our proof of Theorem \ref{th-uniq}, one may extend the fixed-point methods in~\cite{FS} and \cite[Section~XII.5]{G}, to construct perturbative solutions to \eqref{NSE} with $\lambda \neq 0$. The main idea is to view \eqref{NSE} as a perturbation of the linear Oseen system with external force $\mathbf{f}$, and the nonlinear solution will be found close to the Oseen solution in the Banach space $X$ defined in \eqref{eq-Xspace}. Such a completely perturbative scheme would require a condition stricter than scenario (I) and milder than that of Theorem~\ref{th-uniq}. Theorem \ref{th-uniq} can be viewed as a weak-strong uniqueness theorem, in the sense that the perturbative solutions are unique in the broader class of $D$-solutions.

\subsection{Main tools: two new estimates for general Navier-Stokes solutions}

The main tools for our research here are two new estimates on the difference between mean values of the velocity over two concentric circles in terms of the Dirichlet integral in the annulus between them (see Theorems~\ref{th:estim1} and \ref{th:estim2}).  They are also among the main contributions of this paper. As well-known, it is not possible in general to control the growth of a~function through its Dirichlet integral in large planar domains\footnote{For instance, the function $f(z) = \left(\ln(|z|)\right)^\frac13$ has a finite Dirichlet integral in $\mathbb{R}^2 \setminus B_2$ but diverges at infinity.}. Nevertheless, the~special structures of the Navier--Stokes system create the possibility to derive such very general estimates, whose forms turn out to be rather clear and simple. The first estimate claims that 
 \begin{equation} \label{in:estim-m1}
 |\bar\we(r_1)-\bar\we(r_2)|\le C_*(1+\mu)\sqrt{D},
 \end{equation}
 where
 \begin{equation} \label{in:estim-m2}
\mu=\frac1{r_1\me},\ \ \quad \me:=\max\bigl\{|\bar\we(r_1)|,|\bar\we(r_2)|\bigr\},\qquad D:=\int\limits_{r_1\le|z|\le r_2}|\nabla\we|^2
 \end{equation}
 and $C_*$ is some universal positive constant (does not depend on $\we, r_i$, etc.)
The proof of it based on classical methods in \cite{GW, Amick} as well as the recent progress in  \cite{KPR20}.  Namely, by \cite{GW} the pressure is under control of the Dirichlet integral. So, assuming that the estimate~(\ref{in:estim-m1}) fails, we obtain the existence of  two concentric levels sets of the Bernoulli pressure  $\Phi=p+\frac12|\we|^2$, and the difference of the values of $\Phi$ on these level sets is much bigger than $m(1+\mu)\sqrt{D}$. Recall, that 
 \begin{equation} \label{in:estim-m2}\nabla\Phi=-\nabla^\bot\omega+\omega\we^\bot, \end{equation} 
where $\omega=\partial_2w_1-\partial_1w_2$ is the corresponding vorticity. By results of~\cite{GW}, the line integrals of the first term $\nabla\omega$ are small with respect to~$D$, moreover, the variation of {\it the direction} of the velocity $\we$ between our $\Phi$-level sets is under control of the Dirichlet integral as well. The crucial fact is that the vorticity $\omega$ does not change sign between levels sets of $\Phi$ (it was proved in~\cite{KPR20} based on the elegant ideas of~\cite{Amick}). So ~{\it the direction} of $\nabla\Phi$ is almost constant between the concentric level sets of~$\Phi$,  a~contradiction.  

The proof of of the second (asymptotic) estimate crucially relies on the recent results in \cite{KR21new} concerning solutions to the Euler system in the~unit disk with constant velocity on the boundary which are produced by a blow-down procedure (see Section~\ref{sec:3} of the present paper). Namely, it turns out that the pressure satisfies the~classical Neumann boundary conditions, so the absolute value of the pressure for the considered Euler solutions is under control of the Dirichlet integrals as well. 

Both estimates are invariant with respect to the natural rescaling of the stationary Navier--Stokes system.

\subsection{Open problem and discussions}

The key open problem for \eqref{NSE} is to prove existence of solutions in $\mathbb{R}^2$ given arbitrary force (with sufficient decay and regularity) and arbitrary $\lambda$, that is, to remove the constraints in our Theorems \ref{th-s1} and \ref{th-s2}.
 
There are a few works on the construction of solutions to \eqref{NSE} using different approaches from ours. In \cite{GuilW2}, the first author and Wittwer proposed a modified invading domains method which could, for zero total force, produce infinitely many solutions parametrized by their mean values in the unit disc. The spatial limits of their solutions are not clear, \emph{i.e.}, the condition \eqref{NSE}$_3$ is hard to verify. In \cite{Yam09}, Yamazaki constructed solutions with $\lambda = 0$ and explicit decay rates under special symmetry and smallness assumptions on $\mathbf{f}$. In \cite[Section 3]{JG}, the first author also considered the $\lambda = 0$ case and proved existence for $\mathbf{f}$ from a set of codimension three. All three constructions mentioned here work under zero total force assumption. The study for $\lambda \neq 0$ case seems rather limited prior to this work, apart from the possibility of constructing perturbative solutions which we have mentioned earlier. 
 
For $\lambda = 0$, the precise asymptotic behaviour of solutions is of particular interest and difficulty, see, \emph{e.g.}, the discussions in \cite[Section 5]{JG}. Theorem \ref{th-s2} constructs a large class of solutions converging to $0$ at infinity without giving explicit  rates. It would be of great value to prove uniform decay estimates for these solutions. Note that in three dimensions, the Liouville conjecture for $D$-solutions in whole space $\mathbb{R}^3$ with zero spatial limit remains an important open problem where the lack of decay estimates is the key obstacle, see, \emph{e.g.}, \cite[Remark X.9.4]{G} and \cite{Seregin}.

For comparison, we mention that the existence problem for the exterior problem \eqref{NSE-ext} with arbitrary $\lambda$ is also open, and is listed by Yudovich as one of ``eleven great problems of mathematical hydrodynamics" \cite{Y}. See the work of Galdi \cite{G2} for a conditional result on the existence of solutions with large $\lambda$.

Note, in conclusion, that the very recent paper~\cite{ABC} on non-uniqueness of the classical Leray--Hopf  solutions to the non-stationary forced Navier--Stokes system demonstrates the importance and fruitfulness of the studying of such forced equations in the~whole space. 

\subsection{Organization of the paper}

In Section \ref{sec-2}, we introduce the notations, function spaces to be used throughout the paper, along with some useful lemmas for stationary Navier-Stokes solutions. In Section \ref{sec-es1}, we state and prove the first basic estimate for solutions on an arbitrary annulus domain. In Section \ref{sec-es2}, we state and prove the second basic (asymptotic) estimate for a sequence of solutions on enlarging annulus domains. In Section \ref{sec-unif}, we show the uniform bounds for the invading domain solutions $\we_k$  based on the first basic estimate. As a corollary, we prove existence of Leray's solutions $\we_L$. In Section \ref{sec-limit}, we justify \eqref{des-eq} in two scenarios based on the second basic estimate. Finally, in Section \ref{sec-unique}, we state and prove the uniqueness result in a perturbative regime.

\section{Preliminaries} \label{sec-2}

\subsection{Notations} \label{sec-not}
\emph{Throughout the paper, $\mathbf{w}_k$ will be the invading domain solutions to \eqref{NSE-k} for some sequence of radii $R_k \to \infty$.} 

 We use $C$ to denote constants that are independent of $k, R,\lambda,\mathbf{f}$, etc. The exact values of $C$ may change from line to line.

We always assume that $\mathrm{supp} \, \mathbf{f} \subset  B_R$ and $\mathbf{f} \in H^{-1}(B_{2R})$. Definition and some properties of the $H^{-1}$ space are summarized below.

\subsection{Function spaces} \label{sec-fs}

We use standard notations for Sobolev spaces $H^1(\Omega)=W^{1,2}(\Omega)$. Let $\dot{H}_0^1(\Omega)$ be the completion of $C_c^\infty(\Omega)$ in $H^1(\Omega)$ as usual, with norm given by
\begin{equation}\label{sob1}
\|\varphi\|_{\dot{H}_0^1(\Omega)}^2 = \|\nabla \varphi\|^2_{L^2(\Omega)}.
\end{equation}
We recall the following elementary fact:
\begin{lemma}
For any function $\varphi\in H^1(B_R)$ and for every $q\in(1,\infty)$ the inequality
\begin{equation}\label{scale3}
\|\varphi\|_{L^q(B_R)}\le C \biggl(|\bar \varphi(R)| +\|\nabla \varphi\|_{L^2(B_R)}\biggr)R^{2/q}
\end{equation}
holds, where $\bar{\varphi}(R)$ is the average of $\varphi$ on the circle $S_R$, and $C$ is a constant depending only on $q$.
\end{lemma}

\begin{proof}
By scaling, the statement can be reduced to the case $R=1$. Then, the estimate follows from \cite[Theorem~1.1.16]{maz'ya}. 
\end{proof}

The $H^{-1}$ norm of a (scalar valued) distribution $\psi$ in $B_{2R}$ is defined as
\begin{equation}
\|\mathbf{\psi}\|_{H^{-1}(B_{2R})} := \sup_{\varphi \in C_0^\infty(B_{2R}), \, \|\varphi\|_{\dot{H}_0^1(B_{2R}) } = 1 } \langle \psi, \varphi \rangle.
\end{equation} 
Let $\sigma>0$. It is easy to check, that for the function $\varphi_\sigma(z):=\sigma^3\varphi(\sigma z)$ we have the following scaling identity:
\begin{equation}\label{scale1}
\|\varphi_\sigma\|_{H^{-1}(B_{2R/\sigma})}=\sigma\|\varphi\|_{H^{-1}(B_{2R})}.
\end{equation}

Let $\chi$ be a smooth cut-off function satisfying $\chi \equiv 1 \mathrm{\ on\ } B_R$ and $\mathrm{supp} \, \chi \subset B_{2R}$, and define the \emph{total force} $\mathcal{F} := \langle \mathbf{f}, \chi \rangle$. Since $\chi$ can be extended up to the function from $C^\infty_0(B_{2R})$ satisfying $|\nabla\chi|\le C\frac1R$, clearly, we have 
\begin{equation}\label{scale2}
|\mathcal{F}| \le C \|\mathbf{f}\|_{H^{-1}(B_{2R})},
\end{equation}
where $C$ is some universal constant (does not depend on~$R$). 
For simplicity, we will formally write $\mathcal{F} = \int_{\mathbb{R}^2} \mathbf{f}$ although it is possible that $\mathbf{f} \notin L^1_{\mathrm{loc}}$. Next, we state a result in the spirit of \cite[Lemma 3.6]{GuilW2}:
\begin{lemma} \label{lem-tensorF}
Under the above notations, if $\mathcal{F} = 0$, then there exists a tensor $\mathbb{F} \in L^2(\mathbb{R}^2)$ such that $\mathrm{supp}\, \mathbb{F} \subset B_{2R}$, $\mathbf{f} = \nabla \cdot \mathbb{F}$ on $\mathbb{R}^2$, and
$$\|\mathbb{F}\|_{L^2} \le C \|\mathbf{f}\|_{H^{-1}(B_{2R})}$$
with $C$ independent of $R$.
\end{lemma}
\begin{proof}
Let us define $\mathbb{F}$ to be $\nabla \mathbf{g}$ with $\mathbf{g} \in H^1({B_{2R}})$ solving
\begin{equation*} 
\left\{
\begin{aligned}
 & \Delta \mathbf{g} = \mathbf{f},\quad \mathrm{in} \ \ B_{2R},  \\
 & \partial_n \mathbf{g} = 0, \quad \mathrm{on} \ \ \partial B_{2R}.
\end{aligned}
\right.
\end{equation*}
Such $\mathbf{g}$ exists uniquely up to adding constants since we assumed $\mathcal{F} = 0$. We extend $\mathbb{F}$ by $0$ outside $B_{2R}$. Then using $\partial_n \mathbf{g} = 0$ on $\partial B_{2R}$, one can check that $\nabla \cdot \mathbb{F} = \mathbf{f}$ holds in the sense of distributions not only on $B_{2R}$, but also on $\mathbb{R}^2$.
\end{proof}

The following change-of-domain lemma will be useful in Sections \ref{sec-limit} and \ref{sec-unique}.
\begin{lemma} \label{lem-change-of-domain}
Suppose that $\mathrm{supp} \, f \subset B_R$ and $f \in H^{-1}(B_{2R})$. Then for any $a \ge 2$,
\begin{equation}
	\|f\|_{H^{-1}(B_{aR})} \le C \left(1 + \left(\ln \frac{a}{2}\right)^\frac12 \right) \|f\|_{H^{-1}(B_{2R})}
\end{equation}
with $C$ independent of $a$ and $R$. 
If in addition $\mathcal{F}=0$, then
\begin{equation*}
\|f\|_{H^{-1}(B_{aR})} \le C \|f\|_{H^{-1}(B_{2R})}.
\end{equation*}
\end{lemma}
\begin{proof}
Fix the cut-off function $\chi$ as above. For any $\varphi \in \dot{H}_0^1(B_{aR})$, we have $\varphi \chi \in \dot{H}_0^1(B_{2R})$ and $\langle f, \varphi \rangle = \langle f, \varphi \chi \rangle$. Notice that
\begin{align*}
\|\varphi \chi\|_{\dot{H}_0^1(B_{2R})} &\le  \|\chi \nabla \varphi\|_{L^2} + \|\varphi \nabla \chi\|_{L^2} \\
&\le \|\varphi\|_{\dot{H}_0^1(B_{aR})} + \frac{C}{R}\|\varphi\|_{L^2(B_{2R})} \\
&\overset{\eqref{scale3}}{\le}  \|\varphi\|_{\dot{H}_0^1(B_{aR})} + C  \bigg(|\bar \varphi(2R)| +\|\nabla \varphi\|_{L^2(B_{2R})} \bigg) \\
&\le C \|\varphi\|_{\dot{H}_0^1(B_{aR})} \left(1 + \left(\ln \frac{a}{2}\right)^\frac12 \right)
\end{align*}
In the last step we applied the estimate \eqref{eq-new2.9} below to control $|\bar \varphi(2R)| = |\bar \varphi(2R) - \bar \varphi(aR) |$.  Now, a standard duality argument proves the lemma.

If in addition $\mathcal{F} = 0$, then the corresponding claim follows from Lemma \ref{lem-tensorF}, and the fact that for any $a \ge 2$, we have
$$\|\nabla \cdot \mathbb{F}\|_{H^{-1}(B_{aR})} \le \|\mathbb{F}\|_{L^2}.$$
\end{proof}

\begin{lemma} \label{lem-new4}
For any $\phi \in H^1({\Omega}_{\rho_1, \rho_2})$, where $\Omega_{\rho_1, \rho_2} = \{z \in \mathbb{R}^2: 0 < \rho_1 <|z| < \rho_2\}$, we have 
\begin{equation} \label{eq-new2.9}
|\bar{\phi}(\rho_2) - \bar{\phi}(\rho_1)| \le \frac{1}{\sqrt{2\pi}} \left(\int_{\Omega_{\rho_1, \rho_2} } |\nabla \phi|^2 \right)^\frac12 \cdot \left(\ln \frac{\rho_2}{\rho_1}\right)^\frac12
\end{equation}
\end{lemma}
For the proof of Lemma \ref{lem-new4}, see, \emph{e.g.}, \cite[Lemma 2.1]{KPR}.

\subsection{Properties of $D$-solutions}
We present two important lemmas for general $D$-solutions to the Navier--Stokes equations. They have been very useful in many previous studies on the Navier--Stokes exterior problem, see, for instance, \cite{GW, Amick, KPR, KPR20, KR21}. Lemma \ref{lem:pressure} was proved in \cite[Lemma 4.1]{GW}, and Lemma \ref{lem:angle} was proved in \cite[Theorem~4, page~399]{GW}.  Let $\mathbf{w}$ be a $D$-solution to the Navier--Stokes equations in some ring $\Omega_{r_1, r_2} = \{z \in \mathbb{R}^2: 0 < r_1 <|z| < r_2\}$, and $p$ be the corresponding pressure.

\begin{lemma}[\cite{GW}] \label{lem:pressure}
Denote by $\bar{p}(r)$ the average of $p$ over the circle $S_r$. Then for any $r_1 < \rho_1 \le \rho_2 < r_2$, we have
\begin{equation} \label{eq:pdifference}
|\bar{p}(\rho_2) - \bar{p}(\rho_1)| \le \frac{1}{4\pi} \int_{\Omega_{\rho_1, \rho_2}} |\nabla \mathbf{w}|^2.
\end{equation}
\end{lemma}

Denote by $\bar{\mathbf{w}}(r)$ the average of $\mathbf{w}$ over the circle $S_r$. Note, that the direct estimates
for $\bar{\mathbf{w}}(r)$ are not so good:
\begin{equation} \label{eq2a}
|\bar{\mathbf{w}}(\rho_2) - \bar{\mathbf{w}}(\rho_1)| \le \frac{1}{\sqrt{2\pi}} \left(\int_{\Omega_{\rho_1, \rho_2} } |\nabla \mathbf{w}|^2 \right)^\frac12 \cdot \left(\ln \frac{\rho_2}{\rho_1}\right)^\frac12
\end{equation}
(see \eqref{eq-new2.9}). Nevertheless, the direction of $\bar{\mathbf{w}}(r)$ is still under
control of the Dirichlet integral:

\begin{lemma}[\cite{GW}] \label{lem:angle}
  Denote by $\bar{\mathbf{w}}$ the average of $\mathbf{w}$ over the circle $S_r$ and let $\varphi(r) \in [0,2\pi]$ (modulo $2\pi$) be the argument of the complex number associated with the vector $\bar{\mathbf{w}}(r) = (\bar{w}_1(r), \bar{w}_2(r))$, {\emph{i.e.}},
$\bar{\mathbf{w}}(r)=|\bar\we(r)|\,(\cos\varphi(r),\sin\varphi(r))$.  Assume also that $|\bar{\mathbf{w}}(r)| \ge \sigma > 0$ for some constant $\sigma$ and for all $r \in (r_1, r_2)$. Then for any $r_1 < \rho_1 \le \rho_2 < r_2$, we have 
 \begin{equation} \label{eq:phidifference}
  |\varphi(\rho_2) - \varphi(\rho_1)| \le \frac{1}{4\pi \sigma^2} \int_{\Omega_{\rho_1, \rho_2}} \left( \frac{1}{r}|\nabla \omega| + |\nabla \mathbf{w}|^2 \right).
 \end{equation}
 Here $\omega = \partial_2 w_1 - \partial_1 w_2$ is the vorticity of $\mathbf{w}$. 
\end{lemma}

\section{The first basic estimate for the velocity}\label{sec-es1}
We have already mentioned, that the usual estimates for plane functions with finite Dirichlet integral are not efficient to compare the difference of mean values of the function over two circles whose radii are very different (see~(\ref{eq2a})\,). But the special structure of the plane
Navier--Stokes system allows to improve the issue. Here and in the next section we formulate and prove two important estimates for the velocity vector, which will be the main tool for future research in the paper.  Note, that they are valid not only for the solutions in the whole plane (considered above), but in much more general situations. 

\begin{theorem}\label{th:estim1}
Let $\we$ be the $D$-solution to the Navier--Stokes system
 \begin{equation}  \label{NSE-gen} 
\left\{\begin{aligned}
 & \Delta \mathbf{w} - (\mathbf{w} \cdot \nabla) \mathbf{w} - \nabla p = \mathbf{0}, \\
 & \nabla \cdot \mathbf{w} = 0
\end{aligned}
\right.
\end{equation}
 in the annulus type domain $\Omega_{r_1,r_2}=\{z\in\R^2:
 r_1\le|z|\le r_2\}$. Then 
 \begin{equation} \label{eq:estim-m1}
 |\bar\we(r_1)-\bar\we(r_2)|\le C_*(1+\mu)\sqrt{D(r_1,r_2)},
 \end{equation}
 where
 \begin{equation} \label{eq:estim-m2}
\mu=\frac1{r_1\me},\ \ \quad \me:=\max\bigl\{|\bar\we(r_1)|,|\bar\we(r_2)|\bigr\},\qquad D(r_1,r_2):=\int\limits_{\Omega_{r_1,r_2}}|\nabla\we|^2
 \end{equation}
 and $C_*$ is some universal positive constant (does not depend on $\we, r_i$, etc.)
\end{theorem}

In order to prove the theorem, first of all we need to obtain the corresponding estimate for the {\it absolute value} of the velocity.

\begin{lemma}\label{lm:estim1}
Under conditions and notation of Theorem~\ref{th:estim1} the estimate 
 \begin{equation} \label{eq:estim-ab1}
 \biggl| |\bar\we(r_2)|-|\bar\we(r_1)|\biggr|\le C_0(1+\mu)\sqrt{D(r_1,r_2)},
 \end{equation}
holds with some universal positive constant $C_0$  (not depending on $\we, r_i$, etc.)
\end{lemma}

\begin{proof} 
If $|\bar\we(r_2)|=|\bar\we(r_1)|$, we have zero in the left hand side, and there is nothing to prove. For definiteness, assume that $$|\bar\we(r_2)|>|\bar\we(r_1)|$$ (in the case of opposite inequality the~arguments are the same). Let us make several simplifications. First of all, it is sufficient to consider the case
 \begin{equation} \label{f-add1}
\me=|\bar\we(r_2)|=\max\limits_{r\in[r_1,r_2]}|\bar\we(r)|.
 \end{equation}
  Indeed, if the last assumption is not fulfilled, we can take 
 $$r'_2\in[r_1,r_2]\ :\ |\bar\we(r'_2)|=\max\limits_{r\in[r_1,r_2]}|\bar\we(r)|,$$
 and then  consider the interval $[r_1,r'_2]$ instead of $[r_1,r_2]$. 
 By construction, 
 $$\!\!\!D(r_1,r'_2)<D(r_1,r_2),\quad|\bar\we(r_2)|-|\bar\we(r_1)|< |\bar\we(r'_2)|-|\bar\we(r_1)|, \quad\mu(r_1,r'_2)<\mu(r_1,r_2).$$
 So if we prove the required estimate (\ref{eq:estim-ab1}) for $r_1,r'_2$, then it implies the required estimate for initial pair $r_1,r_2$, etc. So below we assume that~(\ref{f-add1}) is fulfilled. 

Further, it is sufficient to consider only the case when
  \begin{equation} \label{f-add2}
|\bar\we(r)|\ge\frac16\me\qquad\forall r\in[r_1,r_2].
 \end{equation}
 Indeed, if the last assumption is not fulfilled, we can take 
 $$r'_1=\max\biggl\{r\in[r_1,r_2]:|\bar\we(r)|\le\frac16\me\biggr\},$$
 and then  consider the interval $[r'_1,r_2]$ instead of $[r_1,r_2]$. 
 By construction, 
  \begin{equation} \label{f-add3}
|\bar\we(r)|\ge\frac16\me\qquad\forall r\in[r'_1,r_2],
 \end{equation}
 $$D(r'_1,r_2)<D(r_1,r_2),\qquad|\bar\we(r_2)|-|\bar\we(r_1)|\le \me=\frac65\biggl(|\bar\we(r_2)|-|\bar\we(r'_1)|\biggr).$$
 So if we prove the required estimate (\ref{eq:estim-ab1}) for $r'_1,r_2$, then it implies the required estimate for initial pair $r_1,r_2$, etc. So below we assume that~(\ref{f-add2}) is fulfilled. 
 
 Now take and fix some small number $\e_0\in(0,\frac1{100})$ (the exact value of $\e_0$ will be specified below). Obviously, it is sufficient to consider only the case when
 \begin{equation} \label{f-add4}
(1+\mu)\sqrt{D(r_1,r_2)}\le \e_0\me.
 \end{equation}
 Indeed, if the opposite inequality were valid, then the required estimate~(\ref{eq:estim-ab1}) is fulfilled automatically with the fixed constant $C_0=\frac1{\e_0}$, and again there are nothing to prove. So below we assume that~(\ref{f-add4}) is fulfilled as well. 
 
Also we can assume without loss of generality that 
$$r_2>1000 r_1$$
(otherwise the required estimate (\ref{eq:estim-ab1}) follows from~(\ref{eq2a})\,). 
Further proof splits into  several steps. Our general strategy is as follows: first of all, we obtain some uniform estimates for pressure, Bernoulli pressure $\Phi=p+\frac12|\we|^2$ and velocity in the suitable annulus type domain; then we deduce some improved estimates for the integral $\int|\nabla\omega|^2$, and finally we obtain the required estimate~(\ref{eq:estim-ab1}) using contradiction argument and level sets of Bernoulli pressure. 
 
 \medskip
 
 \textsc{Step 1.} Using standard estimates (``self-improvement of regularity'') for stationary Stokes system, one can prove that for any $\rho\in[r_1,\frac15r_2]$
  \begin{equation} \label{eq-per1}
\biggl(\int\limits_{\frac32\rho\le r\le\frac92\rho} |\nabla^2\we|^\frac43+|\nabla p|^\frac43\biggr)^{\frac34}
\le C \,m\,\rho^\frac12 D(\rho,5\rho)^\frac12(1+\mu)
 \end{equation} 
 (see, e.g., \cite[Lemma~8 and Appendix~II]{KR21} for the detailed calculations concerning this estimate\,).
 
 \medskip 
 
 \textsc{Step 2.} Now we would like to obtain the uniform estimates for pressure. Using (\ref{eq-per1}), it is easy to see, that 
there exist two ``good'' radii ${\tilde r}_1\in (\frac32r_1,2r_1)$, ${\tilde r}_2\in(\frac12r_2,\frac9{10}r_2)$ such that for $i=1,2$ and  for $D:=D(r_1,r_2)$ we have 
\begin{equation} \label{per-eq3b}
\left\{
\begin{aligned}
&\int_{S_{{\tilde r}_i}} |\nabla \mathbf{w}|^2 ds\le \frac{C D}{{\tilde r}_i}, \\
&\int_{S_{{\tilde r}_i}} |\nabla \mathbf{w}| ds \le (2\pi {\tilde r}_i)^\frac12 \left( \int_{S_{{\tilde r}_i}} |\nabla \mathbf{w}|^2 ds \right)^\frac12  \le C D^\frac12, \\
& \int_{S_{{\tilde r}_i}} |\nabla p|^{\frac43} ds\le C \frac1{{{\tilde r}_i}^\frac13}\,m^\frac43\,D^\frac23(1+\mu)^\frac43, \\
&\int_{S_{{\tilde r}_i}} |\nabla p| ds \le (2\pi {\tilde r}_i)^\frac14 \left( \int_{S_{{\tilde r}_i}} |\nabla p|^\frac43 ds \right)^\frac34  \le C \,m\,D^\frac12(1+\mu).
\end{aligned}
\right.
\end{equation}
Without loss of generality we may assume that $\bar p({\tilde r}_1)=0$. Then by (\ref{eq:pdifference}) we have 
\begin{equation} \label{eq-pr-mean}
|\bar p(r)|\le C D\le C\me\sqrt{D}\qquad\mbox{ for any }r\in[r_1,r_2],
\end{equation} 
thus the estimates~(\ref{per-eq3b}${}_4$) and the trivial inequality
$$|p(z)|\le\bar p(r)+\int\limits_{S_r}|\nabla p|\,ds\qquad\ \forall z\in S_r,$$
imply the corresponding pointwise bound for the pressure:
\begin{equation} \label{per-eq-3-5}
|p|\le C\me (1+\mu) \sqrt{D}\mathrm{\ \ \ \ on\  } S_{{\tilde r}_i}, \ i=1,2.
\end{equation}

From the main estimate~(\ref{eq:pdifference}), using ``good circles'' technique, it is very easy to deduce from (\ref{eq-per1}) and  (\ref{per-eq-3-5}), that  
\begin{equation} \label{per-eq-p1}
|p|\le C\me (1+\mu) \sqrt{D}\mathrm{\ \ \ \ in\  \ } \Omega\bigl(6r_1,\frac16 r_2\bigr)
\end{equation}
(see Appendix~\ref{appendix-1} for details). 

\medskip 
 
 \textsc{Step 3.} Denote $\me_0=|\bar\we(r_1)|$. Recall, that $\me=|\bar\we(r_2)|>m_0$. From (\ref{eq2a}) we have
 \begin{equation} \label{p-corr1}
|\bar\we(r_i)|-C\sqrt{D}\le|\bar\we(\tilde r_i)|\le |\bar\we(r_i)|+C\sqrt{D}, \ i=1,2.
\end{equation}
  In particular,
 \begin{equation} \label{p-corr2}
|\bar\we(\tilde r_1)|\le m_0+C\sqrt{D},
\end{equation}
\begin{equation} \label{p-corr3}
|\bar\we(\tilde r_1)|^2\overset{(\ref{f-add4})}\le m_0^2+C m\sqrt{D}.
\end{equation}
Analogously,
\begin{equation} \label{p-corr5}
m-C\sqrt{D}\le|\bar\we(\tilde r_2)|\le m+C\sqrt{D},
\end{equation}
\begin{equation} \label{p-corr6}
m^2-Cm\sqrt{D}\le|\bar\we(\tilde r_2)|^2\le m^2+C m\sqrt{D}.
\end{equation}
Furthermore,
 \begin{equation} \label{p-corr7}
|\we|\le|\bar\we(\tilde r_1)|  +\int\limits_{S_{\tilde r_1}}|\nabla\we|\,ds
\overset{(\ref{per-eq3b}_2)}\le 
m_0+C\sqrt{D} \mathrm{\ \ \ \ on\  } S_{{\tilde r}_1},
\end{equation}
\begin{equation} \label{p-corr8}
|\we|^2\overset{(\ref{f-add4})}\le m_0^2+C m\sqrt{D} \mathrm{\ \ \ \ on\  } S_{{\tilde r}_1}.
\end{equation}
Analogously,
\begin{equation} \label{p-corr9}
m-C\sqrt{D}\le|\we|\le m+C\sqrt{D}\mathrm{\ \ \ \ on\  } S_{{\tilde r}_2},
\end{equation}
\begin{equation} \label{p-corr10}
m^2-Cm\sqrt{D}\le|\we|^2\le m^2+C m\sqrt{D}\mathrm{\ \ \ \ on\  } S_{{\tilde r}_2}.
\end{equation}
Estimates (\ref{per-eq-3-5}), (\ref{p-corr8}), (\ref{p-corr10})
 imply the following bound for the Bernoulli pressure $\Phi=p+\frac12|\we|^2$\ \,:
\begin{equation} \label{per-eq-3-6}
\Phi\le \frac12\me_0^2+C\me (1+\mu) \sqrt{D}\mathrm{\ \ \ \ on \  } S_{{\tilde r}_1},
\end{equation}
\begin{equation} \label{per-eq-3-6-}
\Phi\le \frac12\me^2+C\me (1+\mu) \sqrt{D}\mathrm{\ \ \ \ on \  } S_{{\tilde r}_2},
\end{equation}
\begin{equation} \label{per-eq-3-7}
\Phi\ge \frac12\me^2-C\me (1+\mu) \sqrt{D}\mathrm{\ \ \ \ on \  } S_{{\tilde r}_2}
\end{equation}
for some universal constant~$C$ (does not depend on~$\we,D,\me,\mu$, etc.). 
By one-sided maximum principle for the Bernoulli pressure (see, e.g., \cite{GW}\,),
\begin{equation} \label{per-eq-3-6+}
\Phi\le \frac12\me^2+C\me (1+\mu) \sqrt{D}\mathrm{\ \ \ \ in \  } \Omega\bigl({\tilde r}_1,{\tilde r}_2\bigr).
\end{equation}
In particular, by (\ref{per-eq-p1})
\begin{equation} \label{per-eq-be}
|\we|^2\le \me^2+C\me (1+\mu) \sqrt{D}\overset{(\ref{f-add4})}\le
C\,m^2 \mathrm{\ \ \ \ in \  } \Omega\bigl(6r_1,\frac16 r_2\bigr),
\end{equation}
in other words,
\begin{equation} \label{per-eq-be2}
|\we|\le C\,\me\mathrm{\ \ \ \ in \  } \Omega\bigl(6r_1,\frac16 r_2\bigr).
\end{equation}
The last inequality gives us the possibility to obtain the improved estimates for~$\nabla\omega$.

\medskip 
 
 \textsc{Step 4.}  We prove that for any $\rho\in[6r_1,\frac1{30}r_2]$
 \begin{equation} \label{f-add5}
\int\limits_{2\rho\le r\le4\rho} |\nabla \omega|^2 \le C\frac1{\rho}\me\,D(\rho,5\rho)(1+\mu).
 \end{equation}
Indeed, for $D_\rho:=D(\rho,5\rho)$ there exist two ``good'' radii $\rho_1\in (\rho,2\rho)$, $\rho_2\in(4\rho,5\rho)$ such that for $i=1,2$, we have 
\begin{equation} \label{eq3b}
\int_{S_{\rho_i}} \omega^2 ds  \le  \frac{C D_\rho}{\rho},
\end{equation}
and
\begin{equation} \label{eq3c}
\left[ \partial_r \int_{S_{r}} \omega^2 ds \right]_{r=\rho_2}
 \le \frac{C D_\rho}{\rho^2},\qquad\qquad -\left[ \partial_r \int_{S_{r}} \omega^2 ds \right]_{r=\rho_1}\le \frac{C D_\rho}{\rho^2},
\end{equation}
The proof of~(\ref{eq3b})--(\ref{eq3c}) is based on some elementary real analysis and is given in the Appendix~\ref{appendix-2} for convenience.

Note that \eqref{per-eq-be2} imply 
\begin{equation} \label{eq3de}
\max\limits_{z\in S_{\rho_i}} |\we(z)|\le C\,\me.
\end{equation}
It is well-known that $\omega$ satisfies the~vorticity equation
\begin{equation} \label{eq:vorticity}
-\Delta \omega + \mathbf{w} \cdot \nabla \omega = 0.
\end{equation}
A standard energy estimate in the domain $\{ \rho_1 \le r \le \rho_2 \}$ for the above equation, together with the bounds \eqref{eq3b}--\eqref{eq3de}, gives
\begin{align*}
\int_{\rho_1 \le r \le \rho_2} |\nabla \omega|^2 &= \int_{S_{\rho_2}} \left(\omega \partial_r \omega  - \mathbf{w} \cdot \mathbf{e}_r \frac{\omega^2}{2}\right) ds \\
&\quad - \int_{S_{\rho_1}} \left(\omega \partial_r \omega  - \mathbf{w} \cdot \mathbf{e}_r \frac{\omega^2}{2}\right) ds \\
&\le C \frac{1}{\rho} D_\rho\biggl(\me+\frac1\rho\biggr)\le C \frac{1}{\rho}\me D_\rho(1+\mu).
\end{align*}

\medskip

 \textsc{Step 5.} Applying (\ref{f-add5}) for $\rho=6r_1$, $12r_1$,\dots, etc., we obtain finally that 
 \begin{equation} \label{f-add8}
\int\limits_{12r_1\le r\le\frac18r_2} r|\nabla \omega|^2 \le C\me D(1+\mu).
 \end{equation}
 
 \medskip

 \textsc{Step 6.}  We need a pair of ``good circles'' with the same properties as in Step~2--3, but now inside the annulus $\Omega(12r_1,\frac18r_2)$ (in order to use estimates~(\ref{f-add8})\,). So, repeating the arguments of these steps, we obtain, that there exists $\tilde\rho_1\in\bigl(12r_1,24 r_1\bigr)$, \ 
 $\tilde\rho_2\in\bigl(\frac1{16}r_2,\frac18r_2\bigr)$ such that
 \begin{equation} \label{eq-3-6}
\Phi\le \frac12\me_0^2+C\me (1+\mu) \sqrt{D}\mathrm{\ \ \ \ on \  } S_{{\tilde \rho}_1},
\end{equation}
\begin{equation} \label{eq-3-7}
\Phi\ge \frac12\me^2-C\me (1+\mu) \sqrt{D}\mathrm{\ \ \ \ on \  } S_{{\tilde \rho}_2}
\end{equation}
for some universal constant~$C$ (does not depend on~$\we,D,\me,\mu$, etc.). 
 Therefore,
\begin{equation} \label{eq-3-8}
\min\limits_{z\in S_{\tilde\rho_2}} \Phi(z)\ge \max\limits_{z\in S_{\tilde\rho_1}} \Phi(z)+\frac12\biggl(\me^2-\me_0^2-C\me (1+\mu) \sqrt{D}\biggr)
\end{equation}

Let $\varphi(r)$ be the direction of the vector $\bar{\mathbf{w}}(r) = (\bar{w}_1(r), \bar{w}_2(r))$, {\emph{i.e.}}, 
$\bar{\mathbf{w}}(r)=|\bar{\mathbf w}(r)|\,(\cos\varphi(r),\sin\varphi(r)).$  Without loss of generality we may assume, that  $\varphi(\tilde\rho_2)=0$. 
Then from the inequality~(\ref{f-add8}) we have
\begin{equation} \label{eq-3-8'}
\int\limits_{12r_1\le r\le\frac18r_2} \frac1r|\nabla \omega| \le \biggl(C\me D(1+\mu)\biggr)^{\frac12}\sqrt{\frac\pi{6 r_1}}\le C \me (1+\mu) \sqrt{D},
 \end{equation}
consequently, from the formula~(\ref{eq:phidifference}) and assumptions~(\ref {f-add2}), (\ref {f-add4}) we obtain that 
 \begin{equation} \label{eq-3-9}
|\varphi(r)|\le C'\e_0\qquad\forall r\in[\tilde\rho_1,\tilde\rho_2]
\end{equation}
 with some universal constant~$C'$ (does not depend on~$\we,D,\me,\mu$, etc.).
 
Without loss of generality we may assume that all the constants $C$ in the inequalities (\ref{eq-3-8})--(\ref{eq-3-8'}) are the same. Now we are ready for the key step of the proof. 
 
\medskip
  
\textsc{Step 7.} We claim that the inequality 
\begin{equation} \label{eq-3-10}
\me-\me_0\le 5C(1+\mu)\sqrt{D}
\end{equation}
holds for sufficiently small~$\e_0$, where $C$ is the same as in~(\ref{eq-3-8})--(\ref{eq-3-8'}). 

Indeed, suppose this claim fails, than 
\begin{equation} \label{eq-3-11}
\me-\me_0> 5C(1+\mu)\sqrt{D},
\end{equation}
in particular, from~(\ref{eq-3-8}) we have 
\begin{equation} \label{eq-3-12}
\min\limits_{z\in S_{\tilde\rho_2}} \Phi(z)> \max\limits_{z\in S_{\tilde\rho_1}} \Phi(z)+2C\me (1+\mu) \sqrt{D}.
\end{equation}
Now, we are in a position to apply the methods of \cite{KPR20} based on level set structures of $\Phi$ to obtain a contradiction when $\varepsilon_0$ is sufficiently small. For a reader's convenience, we recall the main ideas of the proof in~\cite{KPR20} adapted for the present paper. 

From (\ref{eq-3-12}) it follows immediately, that there are two closed regular level sets $S'_{k}$ and $S''_{k}$ of $\Phi$ 
such that:

\begin{enumerate}
\item[(i)] \ $\Phi|_{S'}\equiv t'$, \ \ \ \ \ $\Phi|_{S''}\equiv t''$, \ \ \ \ $t''-t'>2C\me (1+\mu) \sqrt{D}$;
\item[(ii)] \qquad $S'$, $S''$ \ are smooth closed curves (homeomorphic to the circle) surrounding the origin, both of them lie between circles $S_{\tilde\rho_1}$ and $S_{\tilde\rho_2}$.  
\item[(iii)] \qquad the vorticity $\omega(z)$ does not change sign between the curves $S'$, $S''$; for definiteness, we can assume without loss of generality that 
\begin{equation}\label{eq-3-13}
\omega(z)>0\qquad\mbox{ for all $z$ between  $S'$ and $S''$}
\end{equation}
\end{enumerate}
(for the last assertion, see \cite[Step~6]{KPR20}\,). 

Using (\ref{eq-3-9}), it is not difficult to prove for $\e_0$ small enough, that there exists a unit vector $\tilde{\mathbf e}=(\cos\tilde\theta,\sin\tilde\theta)$ such that the segment 
$L=\{\rho\tilde{\mathbf e}:\rho\in[\tilde\rho_1,\tilde\rho_2]\}$ satisfies the following properties: 
\begin{enumerate}
\item[(iv)] \ ${\mathbf w}^\bot(z)\cdot\tilde{\mathbf e}<0$ for any $z\in L$, where we denote $(a,b)^\bot=(-b,a)$;  
\item[(v)] \ $\int\limits_L|\nabla\omega|\,ds< 2C\me (1+\mu) \sqrt{D}$.  
\end{enumerate}
Now take two points $A\in L\cap S'$ \ and \ $B\in L\cap S''$ such that the line segment $[A,B]$ lies between the curves $S'$ and $S''$. 
Recall, that the gradient of the Bernoulli pressure satisfies the identity
$$\nabla\Phi\equiv-\nabla^\bot\omega+\omega\cdot{\mathbf w}^\bot.$$
Then we have 
\begin{equation}
\label{en25}
\begin{array}{c}
t''-t'=\Phi(B)-\Phi(A)=\int\limits_{[A,B]}\nabla\Phi\cdot{\tilde{\mathbf e}}\,ds\\[12pt]
=-\int\limits_{[A,B]}\nabla^\bot\omega\cdot{\tilde{\mathbf e}}\,ds+
\int\limits_{[A,B]}\omega{\mathbf w}^\bot\cdot{\tilde{\mathbf e}}\,dr:=I+II.
\end{array}
\end{equation}
Estimate the~terms~$I$ and $II$ separately. From the above property (v) we have
\begin{equation}
\label{en26}
I< 2C\me (1+\mu) \sqrt{D}.
\end{equation}
On the other hand, from (iii)--(iv) we obtain
\begin{equation}
\label{en27}
II<0.
\end{equation}
Therefore, 
$$t''-t'\le2C\me (1+\mu) \sqrt{D},$$
a contradiction with~(i). 

Of course, the above items (i)--(v) are only short description. In case of interest, a reader can find a detailed justification for all these steps in~\cite{KPR20}.

\medskip

The proof of  Lemma~\ref{lm:estim1} is finished completely. 
\end{proof}

\medskip

{\it Proof of Theorem~\ref{th:estim1}}.  \ We start from the same simplifications as in the proof of Lemma~\ref{lm:estim1}. For definiteness, assume that $|\we(r_2)|\ge|\we(r_1)|$. 
Because of the triangle inequality, it is sufficient to consider the case
 \begin{equation} \label{fg-add1}
\me=|\bar\we(r_2)|=\max\limits_{r\in[r_1,r_2]}|\bar\we(r)|.
 \end{equation}
Also, it is sufficient to consider only the situation when
  \begin{equation} \label{fg-add2}
|\bar\we(r)|\ge\frac16\me\qquad\forall r\in[r_1,r_2]
 \end{equation}
  \begin{equation} \label{fg-add4}
\sqrt{D(r_1,r_2)}=\sqrt{D}\le \frac1{100}\me.
 \end{equation}
(see the explanation in the proof of Lemma~\ref{lm:estim1}). Further we have to repeat all the arguments of the proof of Lemma~\ref{lm:estim1} up to inequality~(\ref{eq-3-8'}). 
Then from the formulas~(\ref{eq:phidifference}), (\ref{eq-3-8'}), and assumptions~(\ref {fg-add2}), (\ref {fg-add4}) we obtain that 
 \begin{equation} \label{g-eq-3-9}
|\varphi(r)|\le C\frac1{\me^2}\biggl(D+\me(1+\mu)\sqrt{D}\biggr)\le C'\frac1{\me}(1+\mu)\sqrt{D}\qquad\forall r\in[\tilde\rho_1,\tilde\rho_2].
\end{equation}
The last formula and (\ref{eq:estim-ab1}) imply easily
\begin{equation} \label{geq:estim-m1}
 |\bar\we(\tilde\rho_1)-\bar\we(\tilde\rho_2)|\le C(1+\mu)\sqrt{D}.
 \end{equation}
Then, by virtue of elementary estimate~(\ref{eq2a})   and inclusions $\tilde\rho_1\in(12r_1,24r_1)$, \ $\tilde\rho_2\in\bigl(\frac1{16} r_2,\frac18 r_2\bigr)$, we obtain the required inequality~(\ref{geq:estim-m1}). 

The proof of Theorem~\ref{th:estim1} is finished completely. $\qed$

\section{The second basic (asymptotic) estimate for the velocity}\label{sec-es2}

For the limiting case (when $r_{i}$ are very large and velocity is ``almost constant'' on the boundary circles) the general estimate of Theorem~\ref{th:estim1} can be improved essentially. 

\begin{theorem}\label{th:estim2}
Let $\we_k$ be a sequence of $D$-solutions to the Navier--Stokes system
 \begin{equation}  \label{NSE-gen-k} 
\left\{\begin{aligned}
 & \Delta \mathbf{w}_k - (\mathbf{w}_k \cdot \nabla) \mathbf{w}_k - \nabla p_k = \mathbf{0}, \\
 & \nabla \cdot \mathbf{w}_k = 0
\end{aligned}
\right.
\end{equation}
 in the annulus type domains $\Omega_{r_{1k},r_{2k}}$. Suppose, in addition, that 
\begin{equation} \label{eq:r-ass}
r_{1k}\to+\infty,\qquad\frac{r_{2k}}{r_{1k}}\to+\infty,
 \end{equation} 
 and  there exist two vectors $\we_0$, $\we_\infty\in\R^n$ such that 
 \begin{equation} \label{e-eq:estim-a1}
 \max_{z\in S_{r_{1k}}}|\we_k(z)-\we_0|\to0, \qquad\max_{z\in S_{r_{2k}}}|\we_k(z)-\we_\infty|\to0\qquad\mbox{\rm as }k\to\infty.
 \end{equation}
 Then 
 \begin{equation} \label{2e-eq:estim-a2}
|\we_0-\we_\infty|\le C_{**}\frac{D_*}{\me},
 \end{equation}
 where $\me:=\max\bigl\{|\we_0|,|\we_\infty|\bigr\}$, $D_* =\varliminf\limits_{k\to\infty}\int_{\Omega_{r_{1k},r_{2k}}}|\nabla\we_k|^2$,
 and $C_{**}$ is some universal positive constant (does not depend on $\we_k$, etc.)
\end{theorem}

To prove this theorem, we have to use some fine properties of solutions to Euler system from~\cite{KR21new}, considered in the next section~\ref{sec:3}.

\subsection{On solutions to Euler equations}
\label{sec:3}

In this section we consider some properties of weak solutions ${\mathbf v}\in W_\loc^{1,2}(\Omega)$ to the Euler system
\begin{equation}
\label{Eul}\left\{\begin{array}{rcl} \big({\bf
v}\cdot\nabla\big){\bf v}+\nabla p & = & 0 \qquad \ \ \
\hbox{\rm in }\;\;\Omega,\\[4pt]
	\div{\bf v} & = & 0\qquad \ \ \ \hbox{\rm in }\;\;\Omega,
\\[4pt]
{\bf v} &  = & {\mathbf e}\ \
 \qquad\  \hbox{\rm on }\;\;S_1=\partial\Omega,
\end{array}\right.
\end{equation}
where $S_1=\{z\in{\mathbb R}^2:|z|=1\}$ is the unit circle and ${\mathbf e}=(1,0)$ is the unit vector. Here $\Omega$ is the unit disk or its complement, \emph{i.e.}, 
\begin{equation}
\label{om1}\Omega=B_1
\end{equation}
or 
\begin{equation}
\label{om2}\Omega=\R^2\setminus B_1.
\end{equation}
Suppose that 
\begin{equation}
\label{de}\int\limits_{\Omega}|\nabla{\mathbf v}|^2<{\varepsilon}^2
\end{equation}
for some constant ${\varepsilon}\in(0,1)$. Then from the first equation~(\ref{Eul}${}_1$) one can assume that $|p|\sim{\varepsilon}$. 
Nevertheless, surprisingly much better estimate holds. 

\begin{theorem} \label{th:euler_est}
Let ${\mathbf v}\in W^{1,2}_{\mathrm{loc}}(\Omega)$ and $p\in W^{1,1}_{\mathrm{loc}}(\Omega)$ satisfy the Euler equations~{\rm (\ref{Eul}${}_{1-2}$)} for almost all~$z\in \Omega$. Suppose also that the~estimate~{\rm(\ref{de})} and the boundary condition~{\rm(\ref{Eul}${}_{3}$)} are fulfilled. Then $p\in C(\bar{\Omega})$, moreover,
\begin{equation}
\label{de-e}\sup\limits_{z_1,z_2\in \bar \Omega}|p(z_1)-p(z_2)|\le C\,{\varepsilon}^2,
\end{equation}
where $C$ is some universal constant (does not depend on ${\varepsilon},{\mathbf v},p$\,).
\end{theorem}

\begin{proof}
For the case $\Omega=B_1$ the result was proved in~\cite{KR21new}. So let us consider the case of exterior domain~$\Omega=\R^2\setminus B_1$. 

By well-known fact concerning $D$-solutions to Euler and
Navier--Stokes system (see, \emph{e.g.}, \cite[Lemma~4.1]{GW}), the
averages of the pressure $\bar p(r)$ are uniformly bounded and have some limit at infinity, without loss of generality we may assume that this limit iz zero:
\begin{equation} \label{co-3}
\bar p(r)\to0\qquad\mbox{ as }r\to\infty.
\end{equation}
Moreover, since
$\int\limits_{\Omega}|\nabla\ve|^2dx<\infty$, there exists an
increasing sequence $r_i\to +\infty$ such that
\begin{equation} \label{co-4}
\int\limits_{S_{r_i}}|\nabla\ve|\,ds\le \frac{\e_i}{\sqrt{\ln r_i}} \qquad\mbox{ with }\e_i\to0\quad\mbox{ as
}i\to\infty
\end{equation}
and
\begin{equation} \label{co-4'}
\sup\limits_{z\in S_{r_i}} \bigl|\ve(z)\bigr|\le \e_i\sqrt{\ln
r_i}
\end{equation}
(see \cite[Lemmas~2.1--2.2]{GW})). From
(\ref{co-3})--(\ref{co-4'}) and from the equation~(\ref{Eul}$_1$)
it follows that
\begin{equation} \label{co-5}
\sup\limits_{z\in S_{r_i}} \bigl|p(z)\bigr|\to0\qquad\mbox{ as }i\to\infty.
\end{equation}
Indeed,
$$
|p(r_i,\theta)-\bar p(r_i)|\le \int\limits_{S_{r_i}}|\nabla p|\,ds
\leq \int\limits_{S_{r_i}}|\ve|\cdot|\nabla \ve|\,ds\leq \e_i\sqrt{\ln
r_i} \int\limits_{S_{r_i}}|\nabla \ve|\,ds\le \e^2_i.
$$

Taking divergence on the first equation in \eqref{Eul} gives
\begin{equation}\label{r2.2}
\Delta p = - \nabla {\mathbf v}\cdot (\nabla{\mathbf v})^\intercal
\end{equation}
{\sl We can extend ${\mathbf v}$ outside $\Omega$ by the constant vector ${\mathbf e}$ so that ${\mathbf v}$ is globally defined and divergence free in $\mathbb{R}^2$.} By the classical div-curl lemma (see, \emph{e.g.}, \cite{CLMS}), \,$\nabla{\mathbf v}\cdot (\nabla{\mathbf v})^\intercal$ belongs to the Hardy space $\mathcal{H}^1(\mathbb{R}^2)$. 
Put
$$
G(x)=-\frac{1}{2\pi}\int\limits_{\Omega} \log|x-y|(\nabla{\mathbf
v}\cdot\nabla{\mathbf v}^\top)(y)\,dy.
$$
By Calder\'on--Zygmund theorem for Hardy's
spaces \cite{St}, \,$G\in D^{2,1}(\R^2)\cap D^{1,2}(\R^2)$,
 where $D^{k,q}(\R^2)$ means the space of measurable functions,whose distributional derivatives of $k$-th order belong to~$L^q(\R^2)$. \ By
classical facts from the theory of Sobolev spaces (see, \emph{e.g.},
\cite{maz'ya}\,), the last inclusion implies that~$G$ is
continuous and convergent to zero at infinity, in particular,
\begin{equation} \label{co-6--}
\|\nabla G\|_{L^2(\R^2)}+\|\nabla^2 G\|_{L^1(\R^2)}\le C\e^2,
\end{equation}
\begin{equation} \label{co-6}
\sup\limits_{z\in\R^2}|G(z)|< C\e^2,
\end{equation}
\begin{equation} \label{Gto0}
G(z) \to 0, \quad \mathrm{as} \ z \to \infty.
\end{equation}
Consider the decomposition 
\begin{equation} \label{decomp0}
p=p_*+G.
\end{equation}
By construction, $\Delta p_*=0$
in $\Omega$, \emph{i.e.}, $p_*$ is a harmonic function, and
by~(\ref{co-5}) we have
\begin{equation} \label{co-7}
\sup\limits_{z\in S_{r_i}}|p_*(z)|\to0\qquad\mbox{ as }i\to\infty.
\end{equation}
Of course, it implies 
\begin{equation} \label{co-7-11}
|p_*(z)|\to0\qquad\mbox{ as }|z|\to\infty.
\end{equation}
Since $p = G + p_*$, we have $p(z) \to 0$ as $z \to \infty$. In \cite[Lemma~9]{KR21new}, it was proved that $p$ satisfies the Neumann boundary conditions (in weak sense) on~$S_1$: 
$$\partial_n p=0\qquad\mbox{ on }S_1.$$
Therefore, $p$ can be solved from \eqref{r2.2} using the Green function for the Poisson problem in $\Omega$  with Neumann boundary conditions:
\begin{equation}\label{p-Neuman}
p(x) = -\frac{1}{2\pi} \int_\Omega \left( \log |x - y| + \log \left|\frac{x}{|x|^2} - y\right| - \log|y| \right) (\nabla{\mathbf
v}\cdot\nabla{\mathbf v}^\top)(y)\, dy.
\end{equation}
It is important to observe that in the above formula, the integral domain $\Omega$ can be replaced by $\mathbb{R}^2$ , since $\nabla{\mathbf
v}\cdot\nabla{\mathbf v}^\top = 0$ in $B_1$ due to our extension of $\mathbf{v}$ by the unit constant vector in~$B_1$. Using the above definition of the function $G(x)$, we can rewrite identity~(\ref{p-Neuman}) as 
$$p(x)=G(x)-G(0)+G\bigl(\frac{x}{|x|^2}\bigr).$$ So the the required bound~(\ref{de-e}) follows from~\eqref{co-6}. 
\end{proof}

The proof of the central Theorem~\ref{th:euler_est} is finished. In \cite[\S2]{KR21new} it was shown, that the established estimates imply 

\begin{corollary} \label{lem:pE}
Let the assumptions of Theorem~\ref{th:euler_est} be fulfilled. Then there exists a set  $\mathscr{F}_0\subset(\frac12,\frac32)$ of positive measure (having one-sided density~$1$ at~$1$\,) \ such that 
\begin{equation} \label{eq:pr-gr7}
S_r\subset\Omega \qquad\mbox{\rm \ and \ \ }\max_{z\in S_r}\biggl(|{\mathbf v}(z)-{\mathbf e}| + \bigl|p(z) - \dashint_{S_{r}} p\bigr|\biggr) \le C {\varepsilon}^2\qquad\qquad\forall r\in\mathscr{F}_0,
\end{equation}where $C$ is some universal constant (does not depend on ${\varepsilon},{\mathbf v},p$, etc.\,).
\end{corollary}

Sometimes solutions to Euler system can be obtained as a~limit of solutions to Navier--Stokes equations when viscosity coefficient tends to zero. We can use this fact in 
the following useful statement.

\begin{theorem}\label{th:estim-eul-lim}
Let the assumptions of Theorem~\ref{th:estim2} be fulfilled.  Suppose, in addition, that $\bar p_k(r_{1k})=0$. Then there exists $\delta\in\bigl(0,\frac12)$ such that 
 \begin{equation}\label{eq:fin-Eul-1}
\me_0\cdot\bigl|{\mathbf w}_k(z)-\we_0\bigr| + \bigl|p_k(z)\bigr| \le C D_*+{\varepsilon}_k\qquad\ \forall z\in S_{(1+\delta)r_{1k}},
\end{equation}
 \begin{equation}\label{eq:fin-Eul-2}
\me_\infty\cdot\bigl|{\mathbf w}_k(z)-\we_\infty \bigr| + \bigl|p_k(z)\bigr| \le C D_*+{\varepsilon}_k\qquad\ \forall z\in S_{(1-\delta)r_{2k}},
\end{equation}
where $\me_\infty=|\we_\infty|$, \ $\me_0=|\we_0|$, $D_*=\varliminf\limits_{k\to\infty}\int_{\Omega_{r_{1k},r_{2k}}}|\nabla\we_k|^2$, \ $\e_k\to0$ as $k\to\infty$, 
 and $C$ is some universal positive constant (does not depend on $\we_k,\,r_{ik}$, etc.)
\end{theorem}

\begin{proof}
The assertion of the last theorem was deduced (by corresponding scaling and limiting procedure) from Theorem~\ref{th:euler_est} and Corollary~\ref{lem:pE} in 
\cite[Section~4]{KR21new} (see, \emph{e.g.}, the formula~(4.10) in \cite{KR21new} and commentaries). 
\end{proof}

\subsection{Proof of Theorem \ref{th:estim2}}\label{sec-es-pr}

Let the conditions of Theorem~\ref{th:estim2} be fulfilled.
As before, we have to make several simplifications. For definiteness, assume that 
$\me=|\we_\infty|\ge|\we_0|$. Denote $\me_0=|\we_0|$. 
Now take and fix some small number $\e_0\in(1,\frac1{100})$ (the exact value of $\e_0$ will be specified below). Obviously, it is sufficient to consider only the case when
 \begin{equation} \label{2-f-add4}
D_*<\e_0\me^2.
 \end{equation}
 Indeed, if the opposite inequality valid, then the required estimate~(\ref{2e-eq:estim-a2}) is fulfilled automatically with the fixed constant $C=\frac2{\e_0}$, and there are nothing to prove. So below we assume that~(\ref{2-f-add4}) is fulfilled as well, therefore,
  \begin{equation} \label{2e-f-add4}
D(r_{1k},r_{2k})<\e_0\me^2
 \end{equation}
 for $k$ large enough. Applying Theorem~\ref{th:estim1} (with $\mu$ arbitrary small as $k\to\infty$ because of assumption $r_{1k}\to+\infty$\,), we obtain
   \begin{equation} \label{2e-f-add2}
\frac12\me<|\bar\we_k(r)|<\frac32\me\qquad\forall r\in[r_{1k},r_{2k}].
 \end{equation}
Take $\delta\in(0,\frac12)$ from Theorem~\ref{th:estim-eul-lim}
and denote $\rho_{1k}=(1-\delta)r_{1k}$,  \ $\rho_{2k}=(1-\delta)r_{2k}$. 
 Repeating Steps~1--6 of the proof of Lemma~\ref{lm:estim1} (with evident modifications and arbitrary small~$\mu$ as $k\to+\infty$), we obtain that 
 \begin{equation} \label{2e-eq-3-8'}
\int\limits_{\rho_{1k}\le r\le\rho_{2k}} \frac1r|\nabla \omega_k| \le \e_k\to0\qquad\mbox{ as \ }k\to+\infty,
 \end{equation}
consequently, from the formula~(\ref{eq:phidifference}) and assumptions~(\ref {2e-f-add2}), (\ref {2e-f-add4}) we obtain that 
 \begin{equation} \label{2e-eq-3-9}
|\varphi_k(r)|\le \frac1{4\pi\me^2} \bigl(D_*+\e_k\bigr)\le \frac1{4\pi}\e_0\qquad\forall r\in[\rho_{1k},\rho_{2k}]
\end{equation}
 for all sufficiently large~$k$. 
 
 Recall, that by Theorem~\ref{th:estim-eul-lim},  
 \begin{equation}\label{2e-eq:fin-Eul-1}
\me_0\bigl| {\mathbf w}_k(z)-\we_0\bigr| + \bigl|p_k(z)\bigr| \le C D_*+{\varepsilon}_k\qquad\ \forall z\in S_{\rho_{1k}},
\end{equation}
 \begin{equation}\label{2e-eq:fin-Eul-2}
\me\bigl|{\mathbf w}_k(z)-\we_\infty\bigr| + \bigl|p_k(z)\bigr| \le C D_*+{\varepsilon}_k\qquad\ \forall z\in S_{\rho_{2k}}.
\end{equation}
Consequently, using $\me_0\ge\frac12\me$, we obtain:
\begin{equation} \label{2e-eq-3-8}
\min\limits_{z\in S_{\rho_{2k}}} \Phi_k(z)\ge \max\limits_{z\in S_{\rho_{1k}}} \Phi_k(z)+\frac12\biggl(\me^2-\me_0^2-12CD_*-\e_k\biggr),
\end{equation}
where $C$ is the same as in~(\ref{2e-eq:fin-Eul-1})--(\ref{2e-eq:fin-Eul-2}). 

Below we consider the case $D_*>0$ (the case $D_*=0$ can be considered similarly, with some evident simplifications). Now we claim that the inequality 
\begin{equation} \label{2e-eq-3-10}
\me-\me_0\le \frac{10}\me CD_*
\end{equation}
holds for sufficiently small~$\e_0$, where $C$ is the same as in~(\ref{2e-eq-3-8}). 
Indeed, suppose that the opposite inequality is valid, then
 from~(\ref{2e-eq-3-8}) we have 
\begin{equation} \label{2e-eq-3-12}
\min\limits_{z\in S_{\rho_{2k}}} \Phi_k(z)> \max\limits_{z\in S_{\rho_{1k}}} \Phi_k(z)+CD_*.\end{equation}
Now we can apply the methods of \cite{KPR20} (based on level set structures of $\Phi_k$) in the same way as on the Step~7 of the proof of previous Lemma~\ref{lm:estim1}, and to obtain the desired contradiction. So the~claim~(\ref{2e-eq-3-10}) is proved.

The formulas (\ref{2e-eq:fin-Eul-1})--(\ref{2e-eq:fin-Eul-2}), (\ref{2e-eq-3-10}) imply that
\begin{equation} \label{2e-eq-3-13}
\bigl| \bar{\mathbf w}_k(\rho_{1k})-\we_0\bigr| \le C\frac1\me D_*\end{equation}
\begin{equation} \label{2e-eq-3-14}
\bigl| \bar{\mathbf w}_k(\rho_{2k})-\we_\infty\bigr| \le C\frac1\me D_*.\end{equation}
\begin{equation} \label{2e-eq-3-15}
\biggl| |\bar{\mathbf w}_k(\rho_{2k})|-|\bar{\mathbf w}_k(\rho_{1k})|\biggr| \le C\frac1\me D_*\end{equation}
with some universal constant~$C$.
We need the following elementary fact: for any pair of vectors $\aa=a\ee_\aa$,\ \ $\bb=b\ee_\bb$, where $\ee_\aa$ and $\ee_\bb$ are the corresponding unit vectors, one has:
$$\bigl|\aa-\bb\bigr|\le \bigl|a-b\bigr|+b\bigl|\ee_\aa-\ee_\bb\bigr|.$$
This elementary formula together with established estimates~(\ref{2e-eq-3-15}), (\ref{2e-eq-3-9}) imply 
\begin{equation} \label{2e-eq-3-16}
\bigl| \bar{\mathbf w}_k(\rho_{2k})-\bar{\mathbf w}_k(\rho_{1k})\bigr| \le C\frac1\me D_*.\end{equation}
Thus by (\ref{2e-eq-3-13})--(\ref{2e-eq-3-14}) the proof of Theorem~\ref{th:estim2} is finished completely. $\qed$

\section{Uniform bounds for the invading domain solutions $\mathbf{w}_k$}\label{sec-unif}

The main result of this section is the following $k$-uniform bounds. 

\begin{theorem} \label{th-uniform-bound}
Suppose that $R > 0$, 
$\mathrm{supp \, } \mathbf{f} \subset B_R$, and $A = \|\mathbf{f}\|_{H^{-1}(B_{2R})}<\infty $, and consider $R_k \ge 2R$. Then the following uniform estimates for the invading domain solutions $\mathbf{w}_k$ to \eqref{NSE-k} are valid: 
\begin{equation} \label{eq31}
D_k:=\int_{{B_{R_k}}} |\nabla \mathbf{w}_k|^2 \le C_1A\biggl(A+\lambda+\frac{A^{\frac13}}{R^{\frac23}}\biggr), 
\end{equation}
\begin{equation} \label{eq31-1}
\max_{R \le r \le R_k} |\bar{\mathbf{w}}_k(r)| \le C_2\biggl(A+\lambda+\frac{A^{\frac13}}{R^{\frac23}}\biggr),
\end{equation} 
where $C_i$ are some universal constants {\rm(}do not depend on $\we_k$ and parameters $A,R,\lambda$, etc.{\rm)}.
\end{theorem}

\begin{remark}
\label{scale-i}
The considered problem~(\ref{NSE-k}) has the following scaling property: if $\mathbf{w}_k$ is a solution to~(\ref{NSE-k}) with parameters $(A,\lambda,R)$, then for any $\tau>0$ the mapping $\we_{\tau k}(z):=\tau\we_k(\tau z)$ is a solution to~(\ref{NSE-k}) with the corresponding parameters 
$(\tau A,\,\tau\lambda,\,\tau^{-1}R)$. Note, that basic estimates in theorems~\ref{th:estim1},\,\ref{th:estim2},\,\ref{th-uniform-bound},\,\ref{th-s1}--\ref{th-s2},\,\ref{th-uniq} are scaling invariant. 
\end{remark}

\begin{proof}
Let $\chi$ be a smooth cut-off function satisfying $\chi \equiv 1 \mathrm{\ on\ } B_R$, $\chi = 0$ outside $ B_{2R}$, and $|\nabla\chi|\le \frac{C}R$. By the system \eqref{NSE-k} and the assumptions on $\mathbf{f}$, we have the following direct estimate of the Dirichlet integral
\begin{align} \label{unif-eq-ee}
D_k = \int_{{B_{R_k}}}|\nabla \mathbf{w}_k|^2 &= \int_{B_{R_k}} \mathbf{f} \cdot (\mathbf{w}_k - \mathbf{w}_\infty)  \nonumber \\
&= \int_{B_{R_k}} \chi \mathbf{f} \cdot (\mathbf{w}_k - \mathbf{w}_\infty) \nonumber \\
&\le \|\mathbf{f}\|_{H^{-1}(B_{2R})} \|\chi(\mathbf{w}_k - \mathbf{w}_\infty)\|_{H_0^1(B_{2R})}  \nonumber \\
&\le C\,A \biggl(\|\nabla\mathbf{w}_k\|_{L^2(B_{2R})} +\frac1R\|\mathbf{w}_k - \mathbf{w}_\infty\|_{L^2(B_{2R})}\biggr)  \nonumber \\
&\le C\,A \bigl( D_k^\frac12 +  |\bar{\mathbf{w}}_k(2R)| + \lambda\bigr),
\end{align}
here we used (\ref{scale3}) for the last inequality
(the exact values of $C$ may change from line to line). 
Using \eqref{unif-eq-ee} it is easy to deduce
\begin{equation} \label{unif-eq-e-uk}
D_k \le C\,A (A+|\bar{\mathbf{w}}_k(2R)| + \lambda).
\end{equation}
Therefore,
\begin{equation} \label{unif-eq-e}
D_k \le C\,A (A+\Lambda_k+ \lambda),
\end{equation}
where we denote
$\Lambda_k =\max_{R \le r \le R_k} |\bar{\mathbf{w}}_k(r)|$. Let $\Lambda_k=|\bar{\mathbf{w}}_k(r_k)|$ with some $r_k\in[R,R_k]$. 
Applying Theorem \ref{th:estim1} in the annulus type domain $\Omega_{r_k, R_k}$ and inserting \eqref{unif-eq-e}, we obtain
\begin{align} \label{unif-eq-1e}
|\Lambda_k - \lambda| &\le C_*(1+ \frac{1}{R\Lambda_k})\sqrt{D_k} \nonumber \\
&\le C (1+ \frac{1}{R\Lambda_k})\sqrt{A (A + \Lambda_k + \lambda)} 
\end{align}
Now we have to consider two different cases:

\medskip
{\sc Case 1.}  $\Lambda_k\le 2(A+\lambda)$. Then the required estimates~(\ref{eq31})--(\ref{eq31-1}) follow immediately.

\medskip
{\sc Case 2.} Now suppose that $\Lambda_k> 2(A+\lambda)$
Then from (\ref{unif-eq-1e}) we obtain:
\begin{equation} \label{unif-eq-w-d3}
 \Lambda_k\le C (1+ \frac{1}{R\Lambda_k})\sqrt{A \Lambda_k} 
\end{equation}
If $\frac1{R\Lambda_k}\le 2$, then, similarly to the Case~1, we have $\Lambda_k\le C\,A$, and the required estimate~(\ref{eq31}) follow easily from~(\ref{unif-eq-e}). So we can assume in addition that~$\frac1{R\Lambda_k}>2$. Then
\begin{equation} \label{unif-eq-w-d4}
 \Lambda_k\le C \frac{1}{R\Lambda_k}\sqrt{A \Lambda_k} \le C \frac{1}{R}\sqrt{\frac{A}{\Lambda_k}}.
\end{equation}
Therefore
\begin{equation} \label{unif-eq-w-d5}
 \Lambda_k\le C \frac{A^{\frac13}}{R^{\frac23}},
\end{equation}
and finally the required estimates~(\ref{eq31})--(\ref{eq31-1}) are~valid in this case as~well.
\end{proof}

\begin{corollary} \label{cor-convergence}
Under the hypothesis of Theorem \ref{th-uniform-bound}, for any sequence $R_k \to \infty$, there exists a subsequence of the invading domain solutions $\mathbf{w}_k$ to \eqref{NSE-k} which converges weakly  to a $D$-solution $\mathbf{w}_L$ to the Navier--Stokes equations \eqref{NSE}$_{1,2}$. Moreover, the convergence is strong in the $C^m_\loc$-topology in $\mathbb{R}^2 \setminus \bar{B}_R$  for any $m \ge 0$, and also strong in $L^2(B_{2R})$.
\end{corollary}

\begin{proof}
Theorem \ref{th-uniform-bound} implies that
\begin{equation} \label{eq-710}
\|\mathbf{w}_k\|_{H^1(B_\rho)} \le C(A, R, \lambda, \rho)
\end{equation}
uniformly in $k$, for any $\rho>0$.  Hence, we can extract a subsequence of $\mathbf{w}_k$ which converges weakly to some function $\mathbf{w}_L \in H^1_{\mathrm{loc}}(\mathbb{R}^2)$ satisfying the equations \eqref{NSE}$_{1,2}$ and 
\begin{equation}
\int_{\mathbb{R}^2} |\nabla \mathbf{w}_L|^2 \le \varliminf\limits_{k\to \infty} \int_{B_{R_k}} |\nabla \mathbf{w}_k|^2 \le C(A, R, \lambda)
\end{equation}
By the standard local regularity theory for stationary Navier--Stokes equations (see, \emph{e.g.}, \cite[Section 2.3]{KR21}), we can obtain uniform local $C^m$ bounds for $\mathbf{w}_k$ in $\mathbb{R}^2 \setminus \bar{B}_R$ (outside the support of $\mathbf{f}$) for any $m > 0$. This together with Arzela-Ascoli theorem implies the local strong $C^m$ convergence in $\mathbb{R}^2 \setminus \bar{B}_R$. Finally, \eqref{eq-710} and the compact embedding $H^1(B_{2R}) \subset L^2(B_{2R})$ implies the strong convergence in $L^2(B_{2R})$.
\end{proof}

\section{The limiting velocity of $\mathbf{w}_L$} \label{sec-limit}

Let $\mathbf{w}_k$ be the invading domain solutions that we studied in the last section. By passing to a subsequence, we assume that $\mathbf{w}_k$ converges weakly to $\mathbf{w}_L$ in the sense described in Corollary \ref{cor-convergence}. In this section, we study the fundamental question: \emph{does $\mathbf{w}_L$ achieve the prescribed limiting velocity $\mathbf{w}_\infty = \lambda \mathbf{e}_1$ at spatial infinity?} 

Since $\mathbf{w}_L$ is a $D$-solution to the Navier--Stokes equations (without force) in the exterior domain $\mathbb{R}^2 \setminus B_{2R}$, it has some finite uniform limit at spatial infinity (this important fact is established through the papers \cite{GW, Amick, KPR}). Let us denote $\mathbf{w}_0 = \lim_{r\to \infty} \mathbf{w}_L$.  We state a simple but crucial identity for the tail Dirichlet integral
\begin{equation} \label{eq-81}
D_* = \lim_{r \to \infty} \overline{\lim_{k \to \infty}}  \int_{\Omega_k \cap \{|z| \ge  r\}} |\nabla \mathbf{w}_k|^2 .
\end{equation}

\begin{lemma} \label{lem-D*}
If $\mathbf{w}_0 \neq 0$, then $D_* = - \mathcal{F} \cdot (\mathbf{w}_\infty - \mathbf{w}_0)$.
\end{lemma}
\begin{proof}
Recall the following energy equalities for the stationary Navier--Stokes solutions $\mathbf{w}_k$ and $\mathbf{w}_L$ respectively,
\begin{equation} \label{eq-82}
D_k = \int_{B_{R_k}} |\nabla \mathbf{w}_k|^2 = \int_{B_{R_k}} \mathbf{f} \cdot (\mathbf{w}_k - \mathbf{w}_\infty) = \int_{B_{R_k}} \mathbf{f} \cdot \mathbf{w}_k  -  \mathcal{F} \cdot \mathbf{w}_\infty,
\end{equation}
\begin{equation} \label{eq-83}
D_L = \int_{\mathbb{R}^2} |\nabla \mathbf{w}_L|^2 = \int_{\mathbb{R}^2} \mathbf{f} \cdot (\mathbf{w}_L - \mathbf{w}_0) = \int_{\mathbb{R}^2} \mathbf{f} \cdot \mathbf{w}_L  -  \mathcal{F} \cdot \mathbf{w}_0.
\end{equation}
The last line relies on the asymptotic behaviour of $\mathbf{w}_L$ at spatial infinity when $\mathbf{w}_0 \neq 0$. By the result of \cite{S}, $\mathbf{w}_L$ is \emph{physically reasonable} in the sense of \cite[page 350]{Sm}. In \cite[Theorem 5]{Sm}  it is proved that $\mathbf{w}_L$ exhibits the wake region behaviour, and has the polynomial convergence rates at spatial infinity indicated by the Oseen fundamental tensor (with some logarithmic correction). Thus \eqref{eq-83} can be obtained with multiplying \eqref{NSE}$_1$ by $\mathbf{w}_L - \mathbf{w}_0$ and integrating on $B_\rho$, then sending $\rho \to \infty$. Observe that $\int \mathbf{f} \cdot \mathbf{w}_k \to \int \mathbf{f} \cdot \mathbf{w}_L$ since $\mathbf{w}_k \rightharpoonup \mathbf{w}_L$ in $H^1(B_{2R})$. It remains to conclude using the simple fact that
\begin{align}
D_* &= \lim_{r \to \infty}  \overline{\lim_{k \to \infty}}  \left\{ D_k -\int_{\Omega_k \cap \{|z| \le  r\}} |\nabla \mathbf{w}_k|^2  \right\} \\
&=  \overline{\lim_{k \to \infty}}  D_k - D_L.
\end{align}
\end{proof}

Next, we justify the equality $\mathbf{w}_0 = \mathbf{w}_\infty$ in two different scenarios.

\subsection{Scenario I : $\lambda$ is much larger than force} \label{sec-81}

\begin{theorem}  \label{th-s1}
There exists an universal constant $\e>0$ such that, 
under the hypothesis of Theorem \ref{th-uniform-bound}, if
 \begin{equation} \label{eq-s1-1}
A\le \frac{\e^2}{\ln^\frac12  \left(\frac{1}{\lambda R} + 2\right)}\lambda,
 \end{equation}
then we have
\begin{enumerate}
\item $\int_{B_{R_k}} |\nabla \mathbf{w}_k|^2 \le C_1 \e^2 \lambda^2$, for $R_k \ge 2\hat{R} := 2\max\{R, \frac{1}{\lambda}\}$,
\item $\max_{\hat{R} \le r \le R_k}|\bar{\mathbf{w}}_k(r)| \le C_2\lambda$,
\item The limiting solution $\mathbf{w}_L$ satisfies the condition $\mathbf{w}_0 = \mathbf{w}_\infty = \lambda \mathbf{e}_1$.
\end{enumerate} 
Here again $C_i$ are some universal constants (do not depend on~$\we_k$, $A$, $\lambda$, etc.). If, in addition, the total force $\mathcal{F} = 0$, then the factor $\ln^\frac12  \left(2+\frac{1}{\lambda R}\right)$ in \eqref{eq-s1-1} can be removed.
\end{theorem}

\begin{proof}

To prove Claims~1--2, we consider two separate cases:

\smallskip
{\bf Case~I: $\lambda R \ge 1$.} In this case, $ \ln^\frac12  \left(\frac{1}{\lambda R} + 2\right)  \sim 1$. The term $\frac{A^{\frac13}}{R^{\frac23}}$ in \eqref{eq31}--\eqref{eq31-1}  can be estimated as
\begin{align}\label{est-sc1-1}
\frac{A^{\frac13}}{R^{\frac23}}=\frac{A^{\frac13}\lambda^{\frac23}}{R^{\frac23}\lambda^{\frac23}}&\le  A^{\frac13}\lambda^{\frac23}  \overset{\eqref{eq-s1-1}}{\le} \e^\frac23 \lambda
\end{align}
Then Claims~1--2 follows immediately from \eqref{eq-s1-1} and Theorem~\ref{th-uniform-bound}.

\smallskip
{\bf Case~II: $\lambda R < 1$.} By Lemma \ref{lem-change-of-domain}, we have 
\begin{equation} 
\|\mathbf{f}\|_{H^{-1}(B_\frac{2}{\lambda})} \le C \left( 1 + \ln ^\frac12 \left(\frac{1}{\lambda R}\right) \right) \|\mathbf{f}\|_{H^{-1}(B_{2R})}
\end{equation}
It is easy to check that, in this case, $ 1 + \ln ^\frac12 \left(\frac{1}{\lambda R}\right)  \sim \ln^\frac12  \left(\frac{1}{\lambda R} + 2\right)$. Since $\frac{1}{\lambda} > R$, $\mathbf{f}$ is of course compactly supported in $B_{\frac{1}{\lambda}}$. Hence, we can reduce to Case I with the new parameters $$R' = \frac{1}{\lambda}$$ and
$$A' \le C \ln^\frac12 \left(\frac{1}{\lambda R} + 2\right) A.$$
Hence, Claims~1--2 are again valid in this case. 

Note that when $\mathcal{F} = 0$, the logarithmic factors can be removed due to the last statement in Lemma \ref{lem-change-of-domain}.

\smallskip
To prove Claim~3, notice that we are now in a position similar to \cite{KR21new}. More precisely,  if $\e<\e_0$ with $\e_0$ small enough, then we have

\smallskip

(a) $D_k = \int |\nabla \mathbf{w}_k|^2  \le \e \lambda^2$,

\smallskip

(b) The total force $|\mathcal{F}| = \left| \int_{\mathbb{R}^2} \mathbf{f} \right| \le \e \lambda$.

\smallskip

These are the key ingredients to invoke the proof of \cite[Theorem 1]{KR21new} (see Section 6 there). Here we can simplify the proof essentially using our second basic estimate in Theorem~\ref{th:estim2}. Namely, using Corollary~\ref{cor-convergence}, we can find $r_{1k} \to \infty$ such that  \eqref{eq:r-ass} and \eqref{e-eq:estim-a1} are valid. By Theorem \ref{th:estim2} we have
 \begin{equation} \label{eq-86}
|\we_0-\we_\infty|\le C_{**}\frac{D_*}{\lambda},
 \end{equation}
where $D_*$ is the tail Dirichlet energy defined in \eqref{eq-81} which controls the limit $\varliminf\limits_{k\to\infty}\int_{\Omega_{r_{1k},r_{2k}}}|\nabla\we_k|^2$. By Lemma \ref{lem-D*} and (a), (b), we have
\begin{equation} \label{eq-87}
D_* \le \begin{cases}
|\mathcal{F}| |\mathbf{w}_\infty - \mathbf{w}_0| \le \e \lambda |\mathbf{w}_\infty - \mathbf{w}_0| , \quad \mathrm{if} \ \mathbf{w}_0 \neq 0, \\
 \varliminf\limits_{k\to \infty} D_k \le \e \lambda^2, \quad \mathrm{if} \ \mathbf{w}_0 = 0.
\end{cases}
\end{equation}
By \eqref{eq-86} -- \eqref{eq-87} we obtain $|\mathbf{w}_\infty - \mathbf{w}_0| \le C_{**} \e|\mathbf{w}_\infty - \mathbf{w}_0|$ which, by taking $\e$ small, immediately implies $\mathbf{w}_\infty = \mathbf{w}_0$.

\end{proof}

\subsection{Scenario II :  zero total force, and $\lambda = 0$}

\begin{theorem}\label{th-s2}
Under the hypothesis of Theorem \ref{th-uniform-bound}, assume in addition that the total force vanishes,   \emph{i.e.}, $\mathcal{F} = \int_{\mathbb{R}^2} \mathbf{f} = \mathbf{0}$. Moreover, we consider zero prescribed limiting velocity, \emph{i.e.}, $\mathbf{w}_\infty = 0$. Then the limiting solution $\mathbf{w}_L$ satisfies the condition $\mathbf{w}_0 = \mathbf{w}_\infty = 0$.
\end{theorem}

\begin{proof}
We assume that $\mathbf{w}_0 \neq 0$ and  obtain a contradiction. By Lemma \ref{lem-D*} we have
\begin{equation} \label{eq-88}
D_* = - \mathcal{F} \cdot (0 - \mathbf{w}_0) = 0
\end{equation}
Using Corollary \ref{cor-convergence}, we can find $r_{1k} \to \infty$ such that  \eqref{eq:r-ass} and \eqref{e-eq:estim-a1} are valid. By Theorem \ref{th:estim2} we have
 \begin{equation} \label{eq-89}
|\we_0-\we_\infty|\le C_{**}\frac{D_*}{|\mathbf{w}_0|},
 \end{equation}
where $D_*$ is the tail Dirichlet energy defined in \eqref{eq-81} which controls the limit $\varliminf\limits_{k\to\infty}\int_{\Omega_{r_{1k},r_{2k}}}|\nabla\we_k|^2$. \eqref{eq-88} -- \eqref{eq-89} immediately imply that $\mathbf{w}_0 = \mathbf{w}_\infty$, a contradiction with the initial assumption. 
\end{proof}

\section{A uniqueness theorem in the case of small force} \label{sec-unique}

\begin{theorem} \label{th-uniq} {
There exists a~universal constant $\e>0$ such that, 
under the hypothesis of Theorem \ref{th-uniform-bound}, if
 \begin{equation} \label{eq-uni-1}
  A< \frac{\e^2}{(1+\lambda R)^3\ln^\frac12  \left(2+\frac{1}{\lambda R}\right)}\lambda,
 \end{equation}
then the problem \eqref{NSE} is uniquely solvable in the class of $D$-solutions. If, in addition, the total force $\mathcal{F} = 0$, then the factor $\ln^\frac12  \left(2+\frac{1}{\lambda R}\right)$ in \eqref{eq-uni-1} can be removed.}
\end{theorem}

\begin{remark}
 In fact, the essence of the proof below can be best understood by setting $R=1, \lambda =1$.
\end{remark}

\begin{proof}
The existence of a $D$-solution $\mathbf{w}_L$ to \eqref{NSE} when $f$ is sufficiently small was already proved in Theorem \ref{th-s1} based on the invading domains method. We just need to show uniqueness. Using scaling, we may assume, without loss of generality, that $\lambda=1$, \emph{i.e.}, $\mathbf{w}_\infty = \mathbf{e}_1$. In particular, the assumption (\ref{eq-uni-1}) will now read
 \begin{equation} \label{eq-uni-0}
  A< \frac{\e^2}{(1+R)^3\ln^\frac12  \left(2+\frac{1}{R}\right)}. \end{equation}
 By the results of \cite{S, Sm}, $\mathbf{w}$ is physically reasonable in the sense of \cite[page 350]{Sm}. In \cite[Theorem 5]{Sm}, it is proved that $\mathbf{w}$ exhibits the wake region behaviour, and has the polynomial convergence rates at spatial infinity indicated by the Oseen fundamental tensor (with some logarithmic correction). Hence, the following energy equality is valid in whole space,
\begin{equation}
D = \int_{\mathbb{R}^2} |\nabla \mathbf{w}|^2 = \int \mathbf{f} \cdot (\mathbf{w} - \mathbf{e}_1).
\end{equation}
Using the ideas of proving Theorems~\ref{th-uniform-bound}, \ref{th-s1}, we obtain the same type of estimate, that is,  
\begin{equation} \label{eq5D}
D \le C A\lambda\le C\frac{\e^2}{(1+R)^3\ln^\frac12  \left(2+\frac{1}{R}\right)}.
\end{equation}
Below we have to consider two different cases:

\medskip
\centerline{\bf Case~I: $R\ge1$.} 
Then by Claim~2 of Theorem~\ref{th-s1} we obtain 
\begin{equation} \label{eq5supw-}
\sup_{r \ge R} |\bar{\mathbf w}(r)| \le C.
\end{equation}
More precisely, using the first basic estimate of Theorem~\ref{th:estim1}, from (\ref{eq5D}) we obtain immediately
\begin{equation} \label{eq5supw}
\sup_{r \ge R} |\bar{\mathbf w}(r)-\ee_1| \le C\frac{\e}{(1+R)^{3/2}}.
\end{equation}
As a by-product in the proof of Theorem~\ref{th:estim1}, from the estimate~(\ref{f-add8}) we also have 
\begin{equation} \label{eq54}
\int_{r \ge \frac32 R } r |\nabla \omega|^2  \le C\frac{\e^2}{(1+R)^3},
\end{equation}
and consequently,
 \begin{equation} \label{eq54-}
\int_{r\ge \frac32 R }|\nabla \omega|^2  \le C\frac{\e^2}{(1+R)^{4}}.
\end{equation}
Next, the proof will be divided into a few steps.

{
\medskip
\textsc{Step 1.}  For simplicity, we denote
\begin{equation}
\eee=R^{-3/2}\e,
\end{equation}
so that
\begin{equation}\label{equniq-0}
D+\int_{r \ge \frac32 R } r |\nabla \omega|^2  \le C\eee^2,
\end{equation}
\medskip
\begin{equation}\label{equniq-00}
\sup_{r \ge R} |\bar{\mathbf w}(r)-\ee_1| \le C\eee.
\end{equation}
By \eqref{equniq-0}--\eqref{equniq-00}, we can find a good circle $S_{r_*}, r_* \in [\frac32R, 2R]$, on which
\begin{equation}\label{equniq-1}
\max\limits_{z\in S_{r_*}}|\we(z)-\ee_1|\le C\eee
\end{equation}
\begin{equation} \label{equniq-1-1}
\max\limits_{z\in S_{r_*}}|p(z) - \bar{p}(r_*)| \le C \eee.
\end{equation}
Without loss of generality, let us assume that $p \to 0$ as $r \to \infty$. Due to Lemma \ref{lem:pressure}, we have 
\begin{equation}
|\bar{p}(r_*)|\le C\eee^2.
\end{equation}
Using the standard decomposition of pressure (see, \emph{e.g.},~(\ref{decomp0}) or  \cite[Corollary 16]{KR21new}), we deduce that 
\begin{equation}\label{equniq-2}
|p(z)|\le C\eee\qquad\mbox{ for }|z|\ge 2R.
\end{equation}
Following \cite[Section 3]{KR21}\footnote{Amick's special $\omega$-level sets as described in \cite[Lemma 13]{KR21} are still available in the whole plane setting outside the support of $\mathbf{f}$.}, we can find a sequence of good radii $r_k \in [2^kR, 2^{k+1}R), \\ k=1,2, 3, \cdots$ such that
\begin{equation}
|\mathbf{w} - \mathbf{e}_1| \le C \eee,
\end{equation}
on each circle $S_{r_k}$. Locating these ``good circles" is a necessary preparation for Step~2.

\medskip
\textsc{Step 2.} Recall Amick's auxiliary function $\gamma = \frac{|\we|^2}{2} + p - \omega \psi$. Here $\psi$ is the stream function satisfying $\nabla \psi = \we^\perp = -w_2 \mathbf{e}_1 + w_1 \mathbf{e}_2$. Let $\psi(2R, 0) = 0$ to be definite. It is known that $\gamma$ satisfies the two-sided maximum principle and converges to $\frac12$ at infinity, see \cite{Amick}. Using \eqref{equniq-1}--\eqref{equniq-1-1}, (\ref{equniq-2}), and the maximum principle for Bernoulli pressure $\Phi = \frac{|\we|^2}{2} + p$, we have
\begin{equation}
|\mathbf{w}(z)| \le C,
\end{equation}
for all $|z| \ge r_*$.
Similarly, using (\ref{equniq-0})--(\ref{equniq-1-1}) and the two-sided maximum principle for $\gamma$  we can deduce that 
\begin{equation}
|\gamma(z) - \frac{1}{2}| \le C R^\frac12 \eee
\end{equation}
for all $|z| \ge 2R$. The $R^\frac12$ factor shows up here because, on a typical circle $S_{r_*}$, $R<r_*<2R$, the term $\omega \psi$ in the definition of $\gamma$ is bounded by $R^{-\frac12} \eee \cdot R = R^\frac12 \eee$. Following \cite[Appendix~I]{KR21}, and using the~information~from Step~1  
we obtain the key pointwise estimate
\begin{equation} \label{amick-prel}
|\mathbf{w}(z) - \mathbf{e}_1| \le C \, R^\frac12 \, \eee=C\, R^{-1}\, \e,
\end{equation}
for all $|z| \ge 2R$.
Also, from~(\ref{scale3}) and (\ref{equniq-0})--(\ref{equniq-00}) we have 
\begin{equation} \label{ballest}
\|\mathbf{w} - \mathbf{e}_1\|_{L^q(B_{3R})}\le C \,R^{\frac2q-\frac{3}{2}} \, \e,
\end{equation}
in particular, 
\begin{equation} \label{ballest-1}
\|\mathbf{w} - \mathbf{e}_1\|_{L^5(B_{3R})}\le C \,R^{-\frac{11}{10}} \, \e.
\end{equation}
}

\medskip
\textsc{Step 3.} Following \cite[Section 5]{KR21}\footnote{The rescaling procedure in \cite{KR21} is not needed here since $\lambda$ is fixed to be 1. However we need to deal with the $R$-dependence in the estimates here while $R=1$ in \cite{KR21}. The big denominator $(1+R)^3$ in (\ref{eq5D}) is useful to absorb all the $R$-dependence that shows up.}, and setting $\e$ sufficiently small, one can obtain the pointwise decay estimate
\begin{equation} \label{eq5amick}
|(\mathbf{w} - \mathbf{e}_1)_i(z)| \le C \e h_i(z), \quad i = 1,2,
\end{equation}
for $|z| \ge 2R$, where $h_i(z)$ is the majorant function given by \footnote{In \cite{FS, KR21}, $h_i(z)$  is defined as $\log \frac{2}{|z|}$ inside the unit disk. Such a logarithmic singularity plays an crucial role there as it accounts for the delicate boundary effects in the flow around obstacle problem. For the whole plane problem, we only use the definition of $h$ outside $B_1$.}
\begin{align}
 |z|>1 &:  \quad \begin{cases}
                   h_1(z) = |z|^{-\frac12},\\
                   h_2(z) = |z|^{-\frac35}.
                  \end{cases}
\end{align}

\textsc{Step 4.} 
Let
\begin{equation}
I_i(z)(\mathbf{u}, \mathbf{v})= \int_{\mathbb{R}^2} \partial_l E_{ij}(z-z') u_j(z') v_l(z') \, dz'_1 dz'_2:=I^R_i(z)(\mathbf{u}, \mathbf{v})+I^\infty_i(z)(\mathbf{u}, \mathbf{v}),
\end{equation}
where 
\begin{equation}
I^R_i(z)(\mathbf{u}, \mathbf{v}) := \int_{B_{2R}} \partial_l E_{ij}(z-z') u_j(z') v_l(z') \, dz'_1 dz'_2,
\end{equation}
\begin{equation}
I^\infty_i(z)(\mathbf{u}, \mathbf{v}) := \int_{\mathbb{R}^2\setminus B_{2R}} \partial_l E_{ij}(z-z') u_j(z') v_l(z') \, dz'_1 dz'_2.
\end{equation}
Here $\mathbf{E}$ is the fundamental Oseen tensor, i.e, the fundamental solution to the Oseen system:
\begin{equation}  
\left\{
\begin{aligned}
 & \Delta E_{ij} - \partial_1 E_{ij} - \partial_i e_j = \delta_{ij} \delta_0, \quad i,j = 1, 2, \\
 & \sum_{i=1,2} \partial_i E_{ij} = 0, \quad j =1,2,
\end{aligned}
\right.
\end{equation}
where $\delta_0$ is the standard $\delta$-function supported at the origin. For more information on $\mathbf{E}$, see for instance \cite[Section 2.4]{KR21} or \cite[Section 2]{GuilW1}.

Recall, that $\partial_l E_{ij}\in L^q(\R^2)$ for any $q\in(3/2,2)$, in particular,
\begin{equation}\label{bile1}
\partial_l E_{ij}\in L^{5/3}(\R^2)
\end{equation}
(for the list of integrability properties for $E_{ij}$, see, \emph{e.g.}, in~\cite{S}). Define the auxiliary norm
\begin{equation}
\|\mathbf{u}\|_Y : = \sup_{|z| \ge 2 R; i=1,2} \frac{|u_i(z)|}{h_i(z)}.
\end{equation}

We will need a crucial bilinear estimate (see \cite[Lemmas~3.1--3.2]{FS}): there exists an absolute constant $C$ such that 
\begin{equation} \label{eq5bilinear}
\sup_{|z| \ge 1; i=1,2} \frac{|I^\infty_i(z)(\mathbf{u}, \mathbf{v})|}{h_i(z)}\le C\,\|\mathbf{u}\|_Y
\|\mathbf{v}\|_Y.
\end{equation} 
In particular, 
\begin{equation} \label{bile2-} \left| \mathbf I^\infty(z)(\mathbf{u}, \mathbf{v}) \right| \le 
                   C\,\|\mathbf{u}\|_Y
\|\mathbf{v}\|_Y|z|^{-\frac12},\qquad |z|\ge1.
\end{equation} 
Further, from (\ref{bile1}) and the elementary estimate $\|\mathbf{u}\|_{L^5(\R^2\setminus B_{2R})}\le C\,R^{-\frac1{10}}\|\mathbf{u}\|_Y$ we have:
\begin{equation} \label{bile2}\left| \mathbf I^\infty(z)(\mathbf{u}, \mathbf{v}) \right|\le
                   C\,\|\mathbf{u}\|_Y
\|\mathbf{v}\|_YR^{-\frac15},\ \qquad |z|\le1
\end{equation} 
(really, the last pointwise estimate is valid for all $z\in\R^2$, but we will use it in the unit ball only). 
Therefore, from~(\ref{bile2-})--(\ref{bile2}) we conclude:
\begin{equation} \label{bile3}
\|\mathbf I^\infty\|_{L^5(B_{2R})}\le C\,\|\mathbf{u}\|_Y
\|\mathbf{v}\|_Y.
\end{equation} 
\begin{equation} \label{bile4}
\|\mathbf I^\infty\|_{Y}\le C\,\|\mathbf{u}\|_Y
\|\mathbf{v}\|_Y.
\end{equation} 
Similarly, from (\ref{bile1})  and H\"older's inequality we have:
\begin{equation} \label{bile5}
|\mathbf I^R(z)(\mathbf{u}, \mathbf{v})|\le 
                   C\,\|\mathbf{u}\|_{L^5(B_{2R})}
                   \|\mathbf{v}\|_{L^5(B_{2R})},\qquad\ \forall z\in\R^2.
\end{equation} 
In particular,
\begin{equation} \label{bile6}
\|\mathbf I^R\|_{L^5(B_{2R})}\le C\,\|\mathbf{u}\|_{L^5(B_{2R})}
\|\mathbf{v}\|_{L^5(B_{2R})}R^{\frac25}.
\end{equation}
It is well-known that 
$|\nabla\mathbf E(z)|\le C\,\frac{1}{|z|}$  for $|z|\ge1$ (see, \emph{e.g.},~\cite[\S2]{Sm}\,). This implies
\begin{equation} \label{bile7}
|\mathbf I^R(z)(\mathbf{u}, \mathbf{v})|\le 
                   C\,\|\mathbf{u}\|_{L^5(B_{2R})}
                   \|\mathbf{v}\|_{L^5(B_{2R})}\frac{R^\frac65}{|z|},\qquad\ \forall z\in\R^2\setminus B_{2R}.
\end{equation} 
Therefore,
\begin{equation} \label{bile9}
\|\mathbf I^R(\mathbf{u}, \mathbf{v})\|_{Y}\le C\,\|\mathbf{u}\|_{ L^5(B_{2R})}
\|\mathbf{v}\|_{L^5(B_{2R})}R^{ \frac45}.
\end{equation} 
Finally, we can summarize all previous estimates in the following form:
\begin{equation}\label{sum-bil1}
\|\mathbf I(\mathbf{u}, \mathbf{v})\|_{X}\le C\,\|\mathbf{u}\|_X \biggl(
\|\mathbf{v}\|_Y +\|\mathbf{v}\|_{L^5(B_{2R})} R^{\frac45}\biggr),
\end{equation}
and also
 \begin{equation}\label{sum-bil2}
\|\mathbf I(\mathbf{u}, \mathbf{v})\|_{X}\le C\,\|\mathbf{v}\|_X \biggl(
\|\mathbf{u}\|_Y +\|\mathbf{u}\|_{L^5(B_{2R})} R^{\frac45}\biggr),
\end{equation}
where we denote 
\begin{equation} \label{eq-Xspace}
\|\mathbf{u}\|_X : =\|\mathbf{u}\|_Y+ \|\mathbf{u}\|_{L^5(B_{2R})}.
\end{equation}

\textsc{Step 5.} The local uniqueness argument we are going to invoke next is similar to that in \cite[Section 7]{FS}. Suppose $\mathbf{w}^{(1)}, \mathbf{w}^{(2)}$ are two $D$-solution to \eqref{NSE}. Denote $\mathbf{u}^{(k)} := \mathbf{w}^{(k)} - \mathbf{e}_1$, $k =1,2$,
then by Steps~2--3 we have 
\begin{equation} \label{eq5Xsmall}
\|\mathbf{u}^{(k)}\|_X\le C \e, \quad k = 1,2,
\end{equation}
\begin{equation} \label{eq5Xsmall--}
\|\mathbf{u}^{(k)}\|_{L^5(B_{2R})}\le C \e R^{-\frac{11}{10}}, \quad k = 1,2.
\end{equation}
Since $\mathbf{u}^{(k)}$ satisfies the system
\begin{equation}
\left\{
\begin{aligned}
&\Delta \mathbf{u}^{(k)}  - \partial_1 \mathbf{u}^{(k)} - \nabla p^{(k)} = \mathbf{u}^{(k)} \cdot \nabla \mathbf{u}^{(k)} + \mathbf{f},  \\
& \nabla \cdot \mathbf{u}^{(k)} = 0,
\end{aligned}
\right.
\end{equation}
and the decay estimate \eqref{eq5Xsmall}, we obtain the following standard representation formula,
$$
\mathbf{u}^{(k)}(z) = \int_{\mathbb{R}^2} \mathbf{E}(z-z') \cdot \left\{ (\mathbf{u}^{(k)} \cdot \nabla) \mathbf{u}^{(k)} + \mathbf{f} \right\}(z') \, dz'_1 dz'_2, \quad  k=1,2.
$$
From this representation we deduce
\begin{align} \label{eq5udiff}
(u_i^{(1)} - u_i^{(2)})(z) &= \int E_{ij}(z-z') (\mathbf{u}^{(1)} \cdot \nabla \mathbf{u}^{(1)} - \mathbf{u}^{(2)} \cdot \nabla \mathbf{u}^{(2)})_j(z') \, dz'_1 dz'_2 \nonumber \\
& = - \int \partial_l E_{ij}(z-z') (u^{(1)}_l u^{(1)}_j - u^{(2)}_l u^{(2)}_j)(z') \, dz'_1 dz'_2 \nonumber\\
& = - I_i(\mathbf{u}^{(1)} - \mathbf{u}^{(2)}, \mathbf{u}^{(2)}) - I_i(\mathbf{u}^{(1)}, \mathbf{u}^{(1)} - \mathbf{u}^{(2)}) .
\end{align}
Using \eqref{eq5Xsmall}--\eqref{eq5Xsmall--} and \eqref{sum-bil1}--\eqref{sum-bil2} we obtain
\begin{align*}
& \quad \  \|\mathbf{u}^{(1)} - \mathbf{u}^{(2)}\|_X\\ &\le C \|\mathbf{u}^{(1)} - \mathbf{u}^{(2)}\|_X \left( \|\mathbf{u}^{(1)}\|_Y + \|\mathbf{u}^{(2)}\|_Y +\|\mathbf{u}^{(1)}\|_{L^5(B_{2R})} R^{\frac45} + \|\mathbf{u}^{(2)}\|_{L^5(B_{2R})} R^{\frac45}\right) \\
&\le C' \e\|\mathbf{u}^{(1)} - \mathbf{u}^{(2)}\|_X,
\end{align*}
which immediately implies that $\mathbf{u}^{(1)} \equiv \mathbf{u}^{(2)}$ when $\e$ is sufficiently small.
\medskip

\centerline{\bf Case~II: $R<1$.} 

\medskip

The proof here is similar to Case II in the proof of Theorem \ref{th-s1}. By Lemma \ref{lem-change-of-domain}, we have 
\begin{equation}
\|\mathbf{f}\|_{H^{-1}(B_2)} \le C \left( 1 + \ln ^\frac12 \left(\frac{1}{R}\right) \right) \|\mathbf{f}\|_{H^{-1}(B_{2R})}
\end{equation}
 Since $R < 1$, $\mathbf{f}$ is of course compactly supported in $B_1$. Hence, we can reduce to Case I with the new parameters $$R' = 1$$ and
$$A' \le C \ln^\frac12 \left(2+\frac{1}{R}\right) A.$$
Finally, in the case that $\mathcal{F} = 0$, the logarithmic factors here can be removed due to the last statement in Lemma \ref{lem-change-of-domain}.

\end{proof}

\appendix

\section{Uniform estimates for the pressure}\label{appendix-1}

In this section we have to prove estimates~(\ref{per-eq3b}) and (\ref{per-eq-p1}) concerning pressure. 
The main idea of ``good circles'' method is quite simple: 
if we have a~nonnegative integrable function $g:\Omega\to\R_+$ on the~annulus domain 
$\Omega=\Omega\bigl(r_0,\beta r_0\bigr)$ with $$\int\limits_\Omega g=\int\limits_{r_0}^{\beta r_0}\biggl(\ \int\limits_{S_r}g\,ds\biggr)\,dr=M,$$ then there exists a~'good' circle $S_{r_*}$ with $r_*\in[r_0,\beta r_0]$ such that 
$$\int\limits_{S_{r_*}}g\le\frac{M}{(\beta-1) r_0}.$$
So the assertion~(\ref{per-eq3b})  follows from~(\ref{eq-per1}) easily (the role of~$g$ is played by~$|\nabla p|^\frac43$ there). 

Further, following exactly the same argument, from the~assertion~(\ref{eq-per1}) we can deduce the existence of sequence of good radii 
$$\rho_1=\tilde r_1,\ \rho_2,\ \rho_3,\ \dots,\ \rho_k=\tilde r_2,$$
such that 
\begin{equation} \label{app-eq3b}
\left\{
\begin{aligned}
&\rho_{i+1}\in\bigl(2\rho_i,3\rho_i\bigr),\quad i=1,\dots,k-1; \\
& \int_{{\rho}_i<r<\rho_{i+1}} |\nabla p|^{\frac43} ds\le C {\rho_i}^\frac23\,m^\frac43\,D^\frac23(1+\mu)^\frac43, \\
& \int_{S_{{\rho}_i}} |\nabla p|^{\frac43} ds\le C \frac1{{\rho_i}^\frac13}\,m^\frac43\,D^\frac23(1+\mu)^\frac43, \\
&\int_{S_{{\rho}_i}} |\nabla p| ds \le (2\pi {\rho}_i)^\frac14 \left( \int_{S_{{\rho}_i}} |\nabla p|^\frac43 ds \right)^\frac34  \le C \,m\,D^\frac12(1+\mu).
\end{aligned}
\right.
\end{equation}
By (\ref{eq-pr-mean}) we have 
\begin{equation} \label{app-pr-mean}
|\bar p(r)|\le C D\le C\me\sqrt{D}\qquad\mbox{ for any }r\in[\rho_1,\rho_k],
\end{equation}
thus the estimates~(\ref{app-eq3b}${}_4$) and the trivial inequality
$|p(z)|\le\bar p(r)+\int\limits_{S_r}|\nabla p|\,ds$ for $z\in S_r,$
imply the corresponding pointwise bound for the pressure:
\begin{equation} \label{app-eq-3-5}
|p|\le C\me (1+\mu) \sqrt{D}\mathrm{\ \ \ \ on\  } S_{{\rho}_i}, \ i=1,\dots,k.
\end{equation}
Now take arbitrary and fix a~point $z_0\in \Omega(\rho_i,\rho_{i+1})$ with $i\in\{2,\dots,k-2\}$. Put $R_0=|z_0|-\rho_{i-1}$ and denote by $B(z_0,R_0)$ the disk centered at~$z_0$ with radius~$R_0$.
Then by construction and by the triangle inequality,
\begin{equation} \label{app-p-nc1}
\begin{aligned}
& B(z_0,R_0)\subset\Omega\bigl(\rho_{i-1},\rho_{i+2}\bigr), \\
&\frac12\rho_{i}<R_0<\frac83\rho_i,\\
& B(z_0,R_0)\cap S_{\rho_i}\ne\emptyset\ne B(z_0,\frac12R_0)\cap S_{\rho_i}.
\end{aligned}
\end{equation}
From~(\ref{app-eq3b}${}_2$) and  (\ref{app-p-nc1}${}_{1-2}$) we have 
\begin{equation} \label{app-p-nc2}
 \int\limits_{B(z_0,R_0)} |\nabla p|^{\frac43} ds\le C {R_0}^\frac23\,m^\frac43\,D^\frac23(1+\mu)^\frac43.
\end{equation}
Therefore, there is a~``good'' circle $S_{z_0,R_*}$ centered at~$z_0$ with radius  $R_*\in\bigl[\frac12 R_0,R_0\bigr]$ such that 
\begin{equation} \label{app-p-nc3}
\left\{
\begin{aligned}
& \int_{S_{z_0,R_*}} |\nabla p|^{\frac43} ds\le C \frac1{{R_*}^\frac13}\,m^\frac43\,D^\frac23(1+\mu)^\frac43, \\
&\int_{S_{z_0,R_*}} |\nabla p| ds \le (2\pi R_*)^\frac14 \left( \int_{S_{z_0,R_*}} |\nabla p|^\frac43 ds \right)^\frac34  \le C \,m\,D^\frac12(1+\mu).
\end{aligned}
\right.
\end{equation}
Since $S_{z_0,R_*}\cap S_{\rho_i}\ne\emptyset$ (see (\ref{app-p-nc1}${}_{3}$)\,), from~(\ref{app-eq-3-5}) and (\ref{app-p-nc3}${}_{2}$) we obtain
\begin{equation} \label{app-p-nc4}
|p|\le C\,m\,D^\frac12(1+\mu)\mathrm{\ \ \ \ on\ \  } S_{z_0,R_*},
\end{equation}
consequently,
\begin{equation} \label{app-p-nc5}
|\bar p(z_0,R_*)|\le C\,m\,D^\frac12(1+\mu),
\end{equation}
where left hand side denotes, as usual, the mean value of~$p$ over the~circle~$S_{z_0,R_*}$. Finally, from the last inequality~(\ref{app-p-nc5}) and from the~basic pressure estimate~(\ref{eq:pdifference}) we obtain immediately 
\begin{equation} \label{app-p-nc6}
|p(z_0)|\le C\,m\,D^\frac12(1+\mu).
\end{equation}
Because of arbitrariness of $z_0$ we have proved the estimate
\begin{equation} \label{app-p-nc7}
|p|\le C\,m\,D^\frac12(1+\mu)\mathrm{\ \ \ \ in\ \  } \Omega(\rho_2,\rho_{k-1})\subset \Omega\bigl(3\tilde r_1,\frac13\tilde r_2\bigr)\subset \Omega\bigl(6r_1,\frac16r_2\bigr)
\end{equation}
as required. 

\

\section{Estimates for the gradient of vorticity}\label{appendix-2}

In this section we have to prove that for any $\rho\in[6r_1,\frac1{30}r_2]$  and $D_\rho:=D(\rho,5\rho)$
there exist two ``good'' radii $\rho_1\in (\rho,2\rho)$, $\rho_2\in(4\rho,5\rho)$ satisfying estimates~(\ref{eq3b})--(\ref{eq3c}).

Denote $f(r)=\int_{S_{r}} \omega^2 ds$. Then by construction $0<f\in C^\infty
\bigl([\rho,5\rho]\bigr)$,
 \begin{equation} \label{appe1}
\int\limits_\rho^{5\rho}f(r)\,dr\le2D_\rho.
 \end{equation}
 Denote 
 $$T=\bigl\{r\in[\rho,5\rho]:f(r)\le \frac8{\rho}D_\rho\bigr\}.$$
 Then by Chebyshev inequality
\begin{equation} \label{appe2}
\meas\biggl([\rho,5\rho]\setminus T\biggr)\le \frac14\rho.
 \end{equation}
 Therefore, there exist points $$t_1\in T\cap\bigl[\rho,\frac43\rho\bigr],\qquad
 t_2\in T\cap\bigl[\frac{14}3\rho,5\rho\bigr].$$ 
 Denote
 $$\tau_1=\max\{r\in[t_1,2\rho]: r\in T\}.$$
 We have to consider two different subcases:
 
 {\bf Case 1}: $\tau_1<2\rho$. Then by construction 
 $f'(\tau_1)\ge0.$ So we put $\rho_1=\tau_1$ and by construction we have 
 \begin{equation} \label{appe3}
\rho_1\in[\rho,2\rho],\qquad f(\rho_1)\le \frac8{\rho}D_\rho,  \qquad f'(\rho_1)\ge0
 \end{equation}

{\bf Case 2}: $\tau_1=2\rho$. Then $\tau_1-t_1\ge\frac23\rho$,  \ $|f(\tau_1)-f(t_1)|\le 
\frac8{\rho}D_\rho$. Therefore by Lagrange mean value theorem there exists $\rho_1\in [t_1,\tau_1]$ with $f'(\rho_1)(\tau_1-t_1)=f(\tau_1)-f(t_1)$, in particular,
$$|f'(\rho_1)|\le \frac{12}{\rho^2}D_\rho.$$

So for both cases 1--2 we find $\rho_1\in[\rho,2\rho]$  satisfying
 \begin{equation} \label{appe4}
\rho_1\in[\rho,2\rho],\qquad f(\rho_1)\le \frac8{\rho}D_\rho,  \qquad -f'(\rho_1)\le
\frac{12}{\rho^2}D_\rho.
 \end{equation}
 
Similarly, we can find $\rho_2\in[4\rho,5\rho]$ satisfying
 \begin{equation} \label{appe5}
\rho_2\in[4\rho,5\rho],\qquad f(\rho_2)\le \frac8{\rho}D_\rho,  \qquad f'(\rho_2)\le
\frac{12}{\rho^2}D_\rho.
 \end{equation}
 
So the required estimates~(\ref{eq3b})--(\ref{eq3c}) are proved completely. \hfill\qed

 \
 
 {\bf Data availability statement.} {\sl Data sharing not applicable to this article as no datasets were generated or analyzed during the current study.}

\

{\bf Conflict of interest statement}. {\sl The authors declare that they have no conflict of interest.}

\bibliographystyle{plain}

\end{document}